\documentclass[10pt,oneside,reqno]{article}
\usepackage[margin=1.3in]{geometry}
\usepackage{amsmath,amsthm,amssymb,color,bm}
\usepackage[authoryear]{natbib}
\usepackage{booktabs}

\usepackage{tikz}
\RequirePackage{amsthm,amsmath,amsfonts,amssymb,mathrsfs,dsfont,mathtools,thmtools}
\RequirePackage[colorlinks,citecolor=blue,urlcolor=blue,linkcolor=blue]{hyperref}
\usepackage[capitalize]{cleveref}
\RequirePackage{graphicx}

\crefname{equation}{}{}
\crefformat{equation}{\textup{(#2#1#3)}}
\crefrangeformat{equation}{\textup{(#3#1#4)--(#5#2#6)}}
\crefdefaultlabelformat{#2\textup{#1}#3}
\crefname{lemma}{Lemma}{Lemmas}
\crefname{page}{p.}{pp.}
\usepackage[normalem]{ulem}

\usepackage{bbm}
\usepackage{mathrsfs}
\usepackage{graphicx}
\usepackage{tikz}
\usepackage{enumitem}
\usetikzlibrary{arrows, automata}
\usetikzlibrary{calc}
\usetikzlibrary{positioning}

\numberwithin{equation}{section}
\allowdisplaybreaks[4]

\theoremstyle{plain}
\newtheorem{theorem}{Theorem}[section]

\newtheorem{lemma}{Lemma}[section]
\newtheorem{corollary}{Corollary}[section]

\newtheorem{remark}{Remark}[section]

\theoremstyle{definition}

\makeatletter

\newcount\minute
\newcount\hour
\newcount\hourMins
\def\now{%
\minute=\time%
\hour=\time \divide \hour by 60%
\hourMins=\hour \multiply\hourMins by 60%
\advance\minute by -\hourMins%
\zeroPadTwo{\the\hour}:\zeroPadTwo{\the\minute}%
}
\def\zeroPadTwo#1{\ifnum #1<10 0\fi#1}

\renewcommand{\cite}{\citet}

\def\^#1{\ifmmode {\mathaccent"705E #1} \else {\accent94 #1} \fi}
\def\~#1{\ifmmode {\mathaccent"707E #1} \else {\accent"7E #1} \fi}

\def\*#1{#1^\ast}
\edef\-#1{\noexpand\ifmmode {\noexpand\bar{#1}} \noexpand\else \-#1\noexpand\fi}
\def\>#1{\vec{#1}}
\def\.#1{\dot{#1}}

\def\atop{\@@atop}
\def\*#1{\mathscr{#1}}

\renewcommand{\leq}{\leqslant}
\renewcommand{\geq}{\geqslant}
\newcommand{\eps}{\varepsilon}
\newcommand{\ep}{\epsilon}
\renewcommand{\eps}{\varepsilon}
\renewcommand{\le}{\leq}

\newcommand{\IE}{\mathbbm{E}}

\newcommand{\Var}{\mathop{\mathrm{Var}}\nolimits}
\newcommand{\Cov}{\mathop{\mathrm{Cov}}}

\newcommand{\IR}{\mathbb{R}}

\def\be#1{\begin{equation*}#1\end{equation*}}
\def\ben#1{\begin{equation}#1\end{equation}}
\def\bes#1{\begin{equation*}\begin{split}#1\end{split}\end{equation*}}
\def\besn#1{\begin{equation}\begin{split}#1\end{split}\end{equation}}

\def\ba#1{\begin{align*}#1\end{align*}}
\def\ban#1{\begin{align}#1\end{align}}

\def\beqn#1\eeqn{\begin{align}#1\end{align}}
\def\beq#1\eeq{\begin{align*}#1\end{align*}}

\usepackage{graphicx}
\usepackage{latexsym}
\usepackage{amsmath,amsthm,amssymb,amscd}
\usepackage{epsf,amsmath}

\def\E{{\IE}}

\newcommand{\red}{\textcolor{red}}

\renewcommand\section{\@startsection {section}{1}{\z@}%
{-3.5ex \@plus -1ex \@minus -.2ex}%
{1.3ex \@plus.2ex}%
{\center\small\sc\mathversion{bold}}}

\def\subsection#1{\@startsection {subsection}{2}{0pt}%
{-3.5ex \@plus -1ex \@minus -.2ex}%
{1ex \@plus.2ex}%
{\bf\mathversion{bold}}{#1}}

\def\subsubsection#1{\@startsection{subsubsection}{3}{0pt}%
{\medskipamount}%
{-10pt}%
{\normalsize\itshape}{\kern-2.2ex. #1.}}

\def\blfootnote{\xdef\@thefnmark{}\@footnotetext}

\makeatother

\begin{document}

\title{Central limit theorem for high temperature spin models via martingale embedding}
\author{Xiao Fang, Yang Xie, Yi-Kun Zhao}
\date{\it The Chinese University of Hong Kong} 
\maketitle

\noindent{\bf Abstract:} 
We use martingale embeddings to prove a central limit theorem (CLT) for one-dimensional projections of high-dimensional random vectors in $\{-1,1\}^n$ satisfying a Poincaré inequality. 
We obtain a non-asymptotic error bound involving two-point and three-point functions for the CLT in 2-Wasserstein distance. We present three illustrative applications: Ising model with finite-range interactions, ferromagnetic Ising model under the Dobrushin condition, and the Sherrington--Kirkpatrick spin glass model at sufficiently high temperature. In all the examples, we allow heterogeneous external fields.

\medskip






\section{Introduction and main result}

Let $X=(X_1,\dots, X_n)^\top$ be a random vector taking values in the $n$-dimensional hypercube $\{-1,1\}^n$. As a prototypical example, we assume that it follows the probability distribution (generalizations beyond quadratic interactions are discussed in \cref{rem:1.1,rem:beyongquadra})
\ben{\label{eq:Isingmodel}
\mu(x):=P(X=x)\propto \exp\left(\frac{1}{2} x^\top A x +h^\top x  \right),\quad x\in \{-1,1\}^n,
}
where $\propto$ means ``is proportional to", $A=(A_{ij})_{1\leq i,j\leq n}$ is a symmetric $n\times n$ matrix (called the \emph{interaction matrix}), and $h\in \IR^n$ is the \emph{external field}. 
The model is \emph{ferromagnetic} if $A_{ij}\geq 0$ for $i\ne j$, but we do not require it in the main theorem.
Note that changing the diagonal entries of $A$ does not change the model \cref{eq:Isingmodel}.
Therefore, by considering $A+a I_n$, where $a\in \IR$ and $I_n$ is the $n\times n$ identity matrix, we assume $A$ is positive semidefinite, i.e., $0\preceq A$ without loss of generality. 
We mostly consider fixed $A$ and $h$. If $A$ is a random matrix as in the Sherrington--Kirkpatrick (SK) spin glass model, our result will be a large probability statement with respect to the randomness of $A$. 

\cite{eldan2022spectral} 
proved that if the operator norm $\|A\|_{\text{op}}$ is less than 1 (a \emph{high-temperature} condition), i.e., $0\preceq A\prec I_n$, then $X$ in \cref{eq:Isingmodel} satisfies the Poincar\'e inequality 
\ben{\label{eq:Poincare}
\Var_\mu(f(X))\leq \frac{1}{1-\|A\|_{\text{op}}} \mathcal{E}_\mu (f, f):= \frac{1}{1-\|A\|_{\text{op}}} \E_\mu \sum_{i=1}^n (\E_\mu[f(X)\vert X_{\sim i}]- f(X))^2,
}
where $X_{\sim i}$ denotes the collection $\{X_j\}_{j\ne i}$ and $\mathcal{E}_\mu (f, f)$ is the \emph{Dirichlet form} corresponding to the continuous time Glauber dynamics (the paragraph above \cref{eq:L2ergo} gives a description of the dynamics). See \cite{bauerschmidt2019very} for an earlier result where the squared term on the right-hand side of \cref{eq:Poincare} is replaced by its upper bound
\be{
(f(X^{\{i,+\}})- f(X^{\{i,-\}}))^2,
}
where $X^{\{i,+\}}$ ($X^{\{i,-\}}$ resp.) has the $i$th coordinate equal to $+1$ ($-1$ resp.) and other coordinates equal to those of $X$.

We are interested in the central limit theorem (CLT) for the total magnetization $\sum_{i=1}^n X_i$. The existing approaches for establishing the CLT include: (a) the \emph{blocking} argument by \cite{newman1980} (see also \cite{goldstein2018}) for positively associated $X_i$'s (the ferromagnetic case); (b) the asymptotic validity of Wick’s law at the level of the four-point function (\cite{newman1975inequalities}, \cite{aizenman1982geometric}, \cite{aizenman2021marginal});
(c) the spatial mixing approach (\cite{kunsch1982decay});
and (d) the exchangeable pair approach in Stein's method (\cite{stein1986approximate}), which works for mean-field Ising models such as the Curie--Weiss model (\cite{chen2013stein}, \cite{deb2023fluctuations}, \cite{lee2025fluctuations}). 
These existing results either assume a spatial structure of the Ising model or a special \emph{linearity condition} in applying Stein's method.

In this paper, we use the martingale embedding method by \cite{eldan2020clt} to prove a CLT for the total magnetization of the Ising model \cref{eq:Isingmodel}. The method will be explained at the beginning of \cref{sec:proofmain}.

Recall the $p$-Wasserstein distance, $p\geq 1$, between two probability distributions $\mu$ and $\nu$ on $\IR$ is defined as
\ben{\label{eq:defpWass}
\mathcal{W}_p(\mu, \nu):= \inf_{X\sim \mu, Y\sim \nu} (\E|X-Y|^p)^{1/p},
}
where the infimum is taken over all couplings of $(X, Y)$ with the corresponding marginal distributions $\mu$ and $\nu$. 
Our first result is as follows:

\begin{theorem}\label{thm:1} 
Let $X$ follow the Ising model \cref{eq:Isingmodel} with $0\preceq A\prec I_n$. Let $\theta=(\theta_1,\dots, \theta_n)^\top\in \IR^n$ be a unit vector, that is, $|\theta|=1$. Let
\ben{\label{eq:Wmusigma}
W_n:=\theta^\top X,\quad \mu_n:=\E(W_n),\quad \sigma_n^2:=\Var(W_n).
}
Then, we have
\besn{\label{eq:thm1}
&\mathcal{W}_2 (\mathcal{L}(W_n), N(\mu_n, \sigma_n^2))\leq 14\left\{\frac{1}{1-\|A\|_{\text{op}}} \sup_{h\in \IR^n} \sum_{k=1}^n \bigg[\sum_{i=1}^n \Big(\sum_{j=1}^n \theta_j B_{ij}^{(h)}\Big)\Big(\sum_{l=1}^n \theta_l B_{ilk}^{(h)}\Big) \bigg]^2 \right\}^{\frac{1}{14}}
}
provided that the right-hand side is sufficiently small, where
\besn{\label{eq:Bilk}
&B_{ij}^{(h)}:=\Cov(X_i, X_j),\\
&B_{ilk}^{(h)}:= \E (X_i-\E X_i)(X_l-\E X_l)(X_k-\E X_k),
}
$N(\mu_n, \sigma_n^2)$ denotes the normal distribution with mean $\mu_n$ and variance $\sigma_n^2$,
the supremum in \cref{eq:thm1} is over all possible external fields $h\in \IR^n$,
and the expectations and covariances in \cref{eq:Bilk} are computed under the model \cref{eq:Isingmodel} with the external field $h$.

\end{theorem}

\begin{remark}[Poincar\'e suffices for the result]\label{rem:1.1}
It can be seen from the proof in \cref{sec:proofmain} that the bound \cref{eq:thm1} is valid for any random vector $X\in \{-1, 1\}^n$ satisfying the Poincar\'e inequality
\ben{\label{eq:rem11}
\Var(f(X))\leq C_p \sum_{i=1}^n\E (f(X^{\{i,+\}})- f(X^{\{i,-\}}))^2,
}
except to change $1/(1-\|A\|_{\text{op}})$ in \cref{eq:thm1} to $C_p\vee 1$.
\end{remark}


We view \cref{eq:thm1} as a covariance-type bound such as those in \cite{newman1980} and \cite{goldstein2018} for the CLT for positively associated random variables and in \cite{barbour1992poisson} for the Poisson approximation under positive or negative association. 
We expect the bound to be small if $\|\theta\|_\infty$ is small and if each spin $X_i$ has a bounded effect on the other spins. 


We provide three illustrative examples in the following: one with finite-range interactions, one in the ferromagnetic case under the Dobrushin condition, and the SK model.

In the following results, we consider a sequence of models indexed by $n$, $n\geq 1$, and the interaction matrix $A=A_n$ and the external field $h=h_n$ in \cref{eq:Isingmodel} depend on $n$. We obtain asymptotic results as $n\to \infty$. Although we may also obtain convergence rates from \cref{eq:thm1}, they are unlikely to be optimal. Hence, we consider only limit theorems.


\subsection{Application 1: Ising model with finite-range interactions}

In the first example, 
let $\Lambda=\Lambda_n$ be a subset of the $d$-dimensional lattice $\mathbb{Z}^d$ of size $|\Lambda|=n$.
Let $\{X_i\}_{i\in \Lambda}$ follow the Ising model \cref{eq:Isingmodel}. Suppose further that the interactions between spins are finite range, that is, $A_{ij}= 0$ if $d(i,j)>r$, where $r$ is a positive integer and $d(i,j):=\max_{k=1,\dots, d} |i_k-j_k|$ is the maximum coordinate-wise distance for $i,j\in \mathbb{Z}^d$.

\begin{corollary}\label{thm:2}
Under the above setting, suppose that the dimension $d$ and range of interaction $r$ are fixed, the interaction matrix $A=A_n$ is positive semidefinite and $\inf_{n\geq 1} (1-\|A_n\|_{op})>0$.
Let 
\be{
W_n=\frac{1}{\sqrt{n}}\sum_{i\in \Lambda}X_i,\quad \mu_n:=\E(W_n), \quad \sigma_n^2:=\Var(W_n).
}
Then,
as $n\to \infty$,
\ben{\label{eq:cor2}
\mathcal{W}_2 (\mathcal{L}(W_n), N(\mu_n,\sigma_n^2)\to 0.
}
If, in addition, $\|h\|_\infty$ is bounded by a universal constant, then
\ben{\label{eq:cor22}
\frac{W_n-\mu_n}{\sigma_n}\to N(0,1)\quad \text{in distribution.}
}
\end{corollary}

The proof of \cref{eq:cor2} follows from \cref{thm:1} and standard arguments for the exponential decay of correlations of fast-mixing Ising models with short-range interactions (\cite{martinelli1999lectures}). For the sake of completeness, we give the proof of \cref{eq:cor2} in \cref{sec:3}. 

To obtain \cref{eq:cor22} from \cref{eq:cor2}, we need $\sigma_n^2$ to be bounded away from 0. 
From the inequality $\Var(X)\geq \E[\Var(X|Y)]$,
the above fact can be argued by selecting order $n$ spins that are conditionally independent given the other spin values and using the condition that the supremum norm of the external field $\|h\|_\infty$ 
is bounded by a universal constant. 

\subsection{Application 2: Ferromagnetic Ising model under the Dobrushin condition}

In the second example, we consider a ferromagnetic Ising model under the Dobrushin condition (\cite{dobrushin1970prescribing}).
Define
\ben{\label{eq:alpha}
\alpha_n:=\sup_{1\leq i\leq n} \sum_{j=1\atop j\ne i}^n|A_{ij}|.
}

\begin{corollary}\label{thm:3}
In the model \cref{eq:Isingmodel}, assume that all the off-diagonal entries of the interaction matrix are nonnegative and the supremum of the Dobrushin constants satisfies $\sup_{n\geq 1}\alpha_n\leq \alpha<1$.
Let 
\be{
W_n=\frac{1}{\sqrt{n}}\sum_{i=1}^n X_i,\quad \mu_n:=\E(W_n), \quad \sigma_n^2:=\Var(W_n).
}
Then, as $n\to \infty$,
\ben{\label{eq:cor3}
\mathcal{W}_2 (\mathcal{L}(W_n), N(\mu_n,\sigma_n^2)\to 0.
}
If, in addition, $\|h\|_\infty$ is bounded by a universal constant, then
\ben{\label{eq:cor32}
\frac{W_n-\mu_n}{\sigma_n}\to N(0,1)\quad \text{in distribution.}
}
\end{corollary}

As far as we know, \cref{thm:3} is the first CLT for Ising models \emph{without} a spatial structure. 
See \cite{kunsch1982decay} for a CLT for the Ising model on the integer lattice under the Dobrushin condition. We provide the proof of \cref{thm:3} in \cref{sec:4}. 

\begin{remark}\label{rem:CW}
The result \cref{eq:cor3} is sharp in a certain sense: for the Curie--Weiss model with zero external field and inverse temperature $\beta>0$,
\be{
\mu(x)\propto \exp\left(\frac{\beta}{2n}\big(\sum_{i=1}^n x_i\big)^2\right),\quad x_i\in \{-1,1\},
}
$\alpha_n$ in \cref{eq:alpha} equals $\frac{n-1}{n}\beta$. It is known that \cref{eq:cor3} does not hold if $\beta\geq 1$ (\cite{ellis2012entropy}).
\end{remark}

\begin{remark}[Beyond quadratic interactions]\label{rem:beyongquadra}
We will see that the proof of \cref{thm:3} works beyond quadratic interactions (we need both $\beta_n$ and $\gamma_n$ in \cref{eq:betagamma} to be bounded away from 1 for the general case). In particular, we can recover the result of \cite{fang2025normal} on the CLT for the exponential random graph model at sufficiently high temperature.


\end{remark}



\subsection{Application 3: SK model}

The Sherrington--Kirkpatrick (SK) spin glass model (\cite{sherrington1975solvable}) is defined by taking the interaction matrix in \cref{eq:Isingmodel} as
$A=\beta H$ with $\beta>0$ and $H$ a $n\times n$ GOE matrix consisting of independent Gaussian entries with variance $1/n$ above the diagonal. The SK model has been a subject of great interest. See \cite{talagrand2010mean,talagrand2010mean2} for an introduction of the model and important techniques (e.g., the cavity method) and results (e.g., the Parisi formula for the limiting free energy). See also recent surveys by \cite{chatterjee2026michel} for the history and by \cite{montanari2026spin} for computational and statistical aspects of the model.

It was pointed out in \cite[Corollary~2]{bauerschmidt2019very} that 
if $\beta<1/4$, then there exists a constant $c_\beta<\infty$ that depends only on $\beta$ such that
the Poincar\'e inequality \cref{eq:rem11} with $C_p=c_\beta$ holds with probability tending to 1 (with respect to the randomness of the GOE $H$) as $n\to \infty$. 
A recent work by \cite{anari2024trickle} improved the threshold from $1/4$ to $\approx 0.295$. It remains an open question whether this bound could be improved up to 1 (\cite{bandeira2025randomstrasse101}).

Applying \cref{thm:1}, we obtain the following result.
\begin{corollary}\label{thm:4}
Let $H_n, n\geq 1$, be a sequence of $n\times n$ GOE matrices as above.
Let $\beta>0$ be a sufficiently small constant.
For each $n$, consider the SK model defined as in \cref{eq:Isingmodel} with the interaction matrix $A=A_n=\beta H_n$.
Let $\langle \cdot \rangle$ denote the conditional expectation with respect to the model given the random interaction matrix.
Let
\be{
W_n=\frac{1}{\sqrt{n}}\sum_{i=1}^n X_i,\quad \mu_n(H_n):=\langle W_n\rangle, \quad \sigma_n^2(H_n):=\langle W_n^2\rangle -\langle W_n\rangle^2.
}
Then, for any $\epsilon>0$, as $n\to \infty$,
\ben{\label{eq:cor4}
P\Big(\mathcal{W}_2 \big(\mathcal{L}(W_n \vert H_n), N(\mu_n(H_n),\sigma_n^2(H_n))\big)\geq \epsilon \Big)\to 0.
}
If, in addition, $\|h\|_\infty$ is bounded by a universal constant $c_h$ and $0<\beta<\beta_h$ for a sufficiently small $\beta_h$ depending on $c_h$, then $\E(\sigma_n^2(H_n))$ is bounded away from 0 for sufficiently large $n$, $\Var(\sigma_n^2(H_n))\to 0$ as $n\to \infty$, and hence \cref{eq:cor4} gives a meaningful CLT.
\end{corollary}

It is possible to trace the explicit threshold for $\beta$ in our proof, and the proof can be improved to optimize the threshold. However, this would require a much longer argument than the current one. Even with such refinements, the threshold is unlikely to reach 0.295, which is itself unlikely to be the optimal Poincar\'e constant (see the discussion above \cref{thm:4}). Therefore, we refrain from pursuing such refinements.



\begin{remark}
CLTs in the SK model have been obtained for the free energy (\cite{aizenman1987some,dey2026fluctuations}), the overlaps (\cite[Section~1.10]{talagrand2010mean}), and general macroscopic observables (\cite{chatterjee2009central}). The modern proofs are typically performed by the cavity method together with the method of moments or Stein's method. 
The result closest to our \cref{thm:4} is \cite[Theorem~1.6]{chatterjee2009central}, which should be able to translate into a CLT for the total magnetization when $\beta$ is sufficiently small and $h_i=h$ for all $i$, although they only did it for the Hamiltonian in their Theorem 1.5(2). 
Their approach is based on Stein's method and relies on the existing results about the 
concentration of the overlaps (which in turn relies on the condition that $h_i=h$ for all $i$). Instead, we proceed by controlling the two-point and three-point functions appearing in our \cref{thm:1}. 
On the other hand, the cavity method will still appear in our proof when we control these functions following the approach of \cite{adhikari2021dynamical} (see \cref{sec:5}).


\end{remark}



\subsection{Further remarks.}

\begin{remark}[On the role of external field]
While we give sharp results in some applications (see \cref{rem:CW}),
one shortcoming of 
the bound \cref{eq:thm1} is that it does not depend on the external field in \cref{eq:Isingmodel}. 
Establishing a Poincar\'e inequality with optimal dependence on the external field is a difficult question in itself.
For example, the main result of \cite{anari2024trickle} depends only on the interaction matrix.
For another example, \cite[Lemma~3.20]{chen2025localization} requires a uniform control on the operator norm of the covariance matrix over all exponential tilts (external fields) of the original measure.
Even if we know that a Poincar\'e inequality holds for the original model, the $\sup_{h\in \IR^n}$ in the bound \cref{eq:thm1} forces us to study the model with arbitrary external field.
The $\sup_{h\in \IR^n}$ in \cref{eq:thm1} comes from the crude bound \cref{eq:step41} and may be replaced by the expectation over a random external field, which depends on the original external field. 
However, the resulting bound appears more complicated and we have not found interesting examples yet.

\end{remark}

\begin{remark}[On the mean and variance]
\cref{thm:1} does not produce (asymptotic) values of the mean $\mu_n$ and variance $\sigma_n^2$. Even the lower bounds on the variance were argued on an ad-hoc basis in our applications. We regard them as separate, problem-specific questions. Note that although the Poincar\'e inequality implies upper bounds on the variance, it does not imply lower bounds. This can be seen by considering the counter-example of the conditional distribution of independent Bernoulli random variables given their summation.
\end{remark}

\begin{remark}[CLT for projections of continuous random vectors]\label{rem:cont}
Our approach also works for continuous random vectors $X$ in $\IR^n$ satisfying a Poincar\'e inequality, that is,
\be{
\Var(f(X))\leq C_p \E |\nabla f(X)|^2
}
for all locally Lipschitz functions $f: \IR^n\to \IR$ with $\E f^2(X)<\infty$, where $\nabla$ denotes the gradient and $|\cdot|$ the Euclidean norm. The approach leads to a similar and slightly more complicated bound than \cref{eq:thm1}.

We would like to mention the related result by \cite{klartag2007central} (see also \cite{klartag2007power}) which states that, for an isotropic, log-concave random vector $X$, $\theta^\top X$ satisfies the CLT for \emph{most} $\theta$ with respect to the uniform probability measure on the unit sphere $S^{n-1}$ in $\IR^n$.  
Related ``randomized" central limit theorems go back to \cite{sudakov1978typical} (see also the recent book by \cite{bobkov2023concentration}).
Recently, the martingale embedding was used in \cite{jiang2020generalized} to study the CLT for the inner product of two log-concave random vectors.
However, we cannot infer a CLT for a given specific $\theta$ from these results. 
\end{remark}

\section{Proof of \cref{thm:1}}\label{sec:proofmain}

Martingale embedding is a well-developed technique to prove CLTs (\cite{hall2014martingale}). Recently, \cite{eldan2020clt} used variations of the martingale embedding constructed by \cite{eldan2016skorokhod} to study the CLT in high dimensions. Their idea is to sum up independent copies of a martingale embedding of a $d$-dimensional random vector and show, by the law of large numbers, that the associated covariance process is well concentrated and thus the resulting sum is close to a $d$-dimensional normal distribution. Our problem is different in that we have a long random vector and we would like to establish the CLT for its projection in a certain direction $\theta$.

\textbf{Proof outline.} In Step 1, we use a particular martingale embedding from \cite{eldan2020clt} (cf. \cref{eq:martingale}) for the $n$-dimensional random vector $X$ following the Ising model \cref{eq:Isingmodel}. Then, the asymptotic normality of $\theta^\top X$ boils down to the concentration of the associated variance process in the direction $\theta$ (cf. \cref{eq:W2}). 
In Step 2, we show that certain parts in the martingale embedding trajectory can be trivially controlled. 
In Step 3, we use the Poincar\'e inequality for $X$ to control the variance of the aforementioned variance process (cf. \cref{eq:afterpoincare}). The resulting upper bound on the variance depends on two-point and three-point functions of $X$ given its value after convoluted with a standard $n$-dimensional Gaussian distribution (cf. \cref{lem:partial}). In Step 4, we observe that the conditional distribution follows the same model \cref{eq:Isingmodel}, except for a change of the external field. As a result, we simplify the final bound to \cref{eq:thm1}.


\medskip

\textbf{Step 1: Martingale embedding.}
We use crucially the following martingale embedding from \cite[Section 4]{eldan2020clt}, which builds upon earlier works by \cite{eldan2016skorokhod} and \cite{eldan2018regularization}.

Let $\mu$ be the probability distribution \cref{eq:Isingmodel} on $\IR^n$.
Let $Y=(Y_t)_{t\in [0,1]}$ be the associated F\"ollmer process with filtration $\mathcal{F}_t$ (\cite{follmer2005entropy,follmer2006time}). 
In particular, the conditional law of $Y$ given the endpoint $Y_1$ is a Brownian bridge.
Set
\ben{\label{eq:XtYt}
X_t=\E[Y_1\vert Y_t]\ \text{and}\  \Gamma_t=\frac{\Cov[Y_1\vert Y_t]}{1-t}.
}
From \cite[Section~4]{eldan2020clt} (see \cref{rem:smooth} for a technical remark), we know that there is a martingale embedding of $X_t$ as
\ben{\label{eq:martingale}
X_t=\E Y_1+\int_0^t \Gamma_t dB_t^{(n)},   \quad t\in [0,1],
}
where $B_t^{(n)}$ is an $\mathcal{F}_t$-adapted standard Brownian motion in $\IR^n$.
In particular, $X_1=Y_1\sim \mu$.
In this proof, $X_t$ and $Y_t$, $0\leq t\leq 1$, denote random vectors in $\IR^n$ with components $X_{t1}, \dots, X_{tn}$ and $Y_{t1},\dots, Y_{tn}$, respectively.
As a consequence, the random variable $W:=W_n$ in \cref{eq:Wmusigma} can be constructed as
\be{
W=\mu_n+\int_0^1 \big|\theta^\top \Gamma_t\big| dB_t,\quad \Var(W)=\int_0^1 \E\big|\theta^\top \Gamma_t \big|^2 dt,
}
where $\big|\theta^\top \Gamma_t\big|=\sqrt{\theta^\top \Gamma_t^2 \theta}$ and $B_t$ is a one-dimensional standard Brownian motion adapted to $\mathcal{F}_t$.

To approximate $W$ by a normal random variable, we rewrite
\ben{\label{eq:couple}
W=\mu_n+\int_0^1 \sqrt{\E |\theta^\top \Gamma_t|^2} dB_t+ \int_0^1\left( \big|\theta^\top \Gamma_t\big|-\sqrt{\E |\theta^\top \Gamma_t|^2} \right) dB_t.
}
Define the normal random variable 
\be{
G=\mu_n+\int_0^1 \sqrt{\E |\theta^\top \Gamma_t|^2} dB_t\sim N(\mu_n, \Var(W)).
}
From the definition of $p$-Wasserstein distance in \cref{eq:defpWass}, the coupling \cref{eq:couple} and the It\^o\kern0pt isometry, we have
\besn{\label{eq:W2}
\mathcal{W}_2^2 (\mathcal{L}(W), \mathcal{L}(G))&\leq \E\left| \int_0^1\left( \big|\theta^\top \Gamma_t\big|-\sqrt{\E |\theta^\top \Gamma_t|^2} \right) dB_t\right|^2\\
&=\int_0^1 \E\left( \big|\theta^\top \Gamma_t\big|-\sqrt{\E |\theta^\top \Gamma_t|^2} \right)^2   dt\\
&\leq \int_0^1  \E \left( \frac{\left|\theta^\top \Gamma_t\right|^2-\E \left|\theta^\top \Gamma_t\right|^2}{\sqrt{\E |\theta^\top \Gamma_t|^2}} \right)^2  dt\\
&= \int_0^1   \frac{\Var(\left|\theta^\top \Gamma_t\right|^2)}{\E |\theta^\top \Gamma_t|^2}   dt.
}

\begin{remark}\label{rem:smooth}
The representation \cref{eq:martingale} is valid as long as $\mu$ has finite second moment (from personal communication with Yuta Koike).
However, \cite{eldan2020clt} only considered the situation that $\mu$ has a smooth density and bounded support.
We may also apply their result directly by convolving $Y_1$ with an arbitrarily small smooth component.
From \cref{eq:W2}, only the joint distribution of $(Y_1, Y_t)$ for any given $t\in [0,1]$ matters in the subsequent computations.
Therefore, by a limiting argument, we can take the small smooth component to be zero from this point onward.
We have also changed their $Y_1\vert \mathcal{F}_t$ to $Y_1\vert Y_t$ in \cref{eq:XtYt} using the Markov property of the F\"ollmer process (the F\"ollmer process used in \cite[Section~4]{eldan2020clt} is defined by a stochastic differential equation with drift as in \cite[Eq.(9)]{eldan2018regularization}).

Alternatively, one may 
use the representation of $X_1$ in \cite[Eq.(18)]{eldan2020stability}, combined with the limiting argument above.
\end{remark}


From the description of the law of $Y$ via Brownian bridge, we have, given $Y_1\sim \mu$, 
the conditional distribution of $Y_t$ is the same as that of $t Y_1+\sqrt{t(1-t)}Z$, where $Z\sim N(0, I_n)$ is an independent $n$-dimensional standard Gaussian vector.

\medskip

\textbf{Step 2: Initial deduction.}
For a small $\eps>0$ to be chosen, 
we only need to control the 2-Wasserstein distance between $\int_{t\in [0,1]: \E |\theta^\top \Gamma_t|^2\geq \eps} \big|\theta^\top \Gamma_t\big| dB_t$ and its Gaussian counterpart. 
In fact, the remaining integral over those $t$ such that $\E |\theta^\top \Gamma_t|^2<\eps$ contributes at most $2\sqrt{\eps}$ to the 2-Wasserstein distance from a variance computation. 
The next lemma shows that the integral over those $t$ too close to 1 is also small.

\begin{lemma}\label{lem:larget}
We have
\be{
\|\theta^\top X_{1-\eps}-\theta^{\top} X_1\|_2\leq 2\sqrt{\frac{\eps}{1-\eps}}.
}
\end{lemma}

\begin{proof}[Proof of \cref{lem:larget}]
Recall from \cref{eq:XtYt} that
\bes{
\theta^\top X_{1-\eps}&= \theta^\top \E[Y_1\vert Y_{1-\eps}]= \E[\theta^\top Y_1\vert (1-\eps) Y_1+\sqrt{\eps(1-\eps)}Z],
}
where $Z\sim N(0, I_d)$ is independent of $Y_1$.
Therefore,
\bes{
&\theta^\top X_{1-\eps}-\theta^{\top} X_1\\
&=\E[\theta^\top Y_1\vert  Y_1+\sqrt{\frac{\eps}{1-\eps}}Z]- \theta^\top Y_1\\
&=\E[\theta^\top Y_1+\sqrt{\frac{\eps}{1-\eps}}\theta^\top Z\vert  Y_1+\sqrt{\frac{\eps}{1-\eps}}Z]- \E[\sqrt{\frac{\eps}{1-\eps}}\theta^\top Z\vert  Y_1+\sqrt{\frac{\eps}{1-\eps}}Z]- \theta^\top Y_1\\
&=\sqrt{\frac{\eps}{1-\eps}}\theta^\top Z- \E[\sqrt{\frac{\eps}{1-\eps}}\theta^\top Z\vert  Y_1+\sqrt{\frac{\eps}{1-\eps}}Z],
}
and hence
\be{
\|\theta^\top X_{1-\eps}-\theta^{\top} X_1\|_2\leq 2\sqrt{\frac{\eps}{1-\eps}}\|\theta^\top Z\|_2=2\sqrt{\frac{\eps}{1-\eps}},
}
where we used the assumption that $\theta$ is a unit vector in the last equality.
\end{proof}
With the initial deduction, 
we have, for $\eps<1/2$,
\ben{\label{eq:deduction}
\mathcal{W}_2 (\mathcal{L}(W), \mathcal{L}(G))\leq 8\sqrt{\eps}+\sqrt{\frac{1}{\eps} \int_0^{1-\eps} \Var(\left|\theta^\top \Gamma_t\right|^2) dt}.
}


\medskip

\textbf{Step 3: Applying Poincar\'e.}
Now we bound
\ben{\label{eq:step31}
\Var(\big|\theta^\top \Gamma_t\big|^2)=\frac{1}{(1-t)^4} \Var\left(\big|\theta^\top \Cov(Y_1\vert Y_t)\big|^2\right).
}
Recall $Y_1$ satisfies the Poincar\'e inequality \cref{eq:Poincare} with Poincar\'e constant 
\be{
C_p:=\frac{1}{1-\|A\|_{\text{op}}}>0.
}
From \cref{eq:Poincare} and the Gaussian Poincar\'e inequality, for any locally Lipschitz functions $f: \IR^n\to \IR$, we have
\ba{
&\Var(f(Y_t))=\Var(f(tY_1+\sqrt{t(1-t)}Z))\\
=&\E \Big(\Var(f(tY_1+\sqrt{t(1-t)}Z) \vert Z)\Big)+\Var\Big(\E (f(tY_1+\sqrt{t(1-t)}Z) \vert Z)\Big)\\
\leq& C_p\sum_{k=1}^n \E \Big(f(tY_1^{\{k,+\}}+\sqrt{t(1-t)}Z)-f(tY_1^{\{k,-\}}+\sqrt{t(1-t)}Z)\Big)^2\\
&+ t(1-t) \sum_{k=1}^n \E \Big(\int \partial_k f(ty+\sqrt{t(1-t)}Z)\mu(dy)   \Big)^2\\
\leq & 4t^2C_p \sum_{k=1}^n \E \Big( \partial_k f(tY_1^{\{k,-\}}+2tUe_k+\sqrt{t(1-t)}Z)\Big)^2\\
&+ t(1-t) \sum_{k=1}^n \E \Big(\partial_k f(tY_1+\sqrt{t(1-t)}Z)   \Big)^2,
}
where $Y_1^{\{k,+\}}$ ($Y_1^{\{k,-\}}$ resp.) has the $k$th coordinate equal to $+1$ ($-1$ resp.) and other coordinates equal to those of $Y_1$, 
$U$ is a uniform random variable in $[0,1]$ independent of everything else, $e_k$ is the unit vector in $\IR^n$ with $1$ in the $k$th coordinate and $0$ in other coordinates, 
and $\partial_k$ denotes the partial derivative with respect to the $k$th coordinate.
This implies that
\ban{\label{eq:afterpoincare}
&\Var\left(\left|\theta^\top \Cov(Y_1\vert Y_t)\right|^2\right)=
\Var\left(\sum_{i=1}^n \left[\sum_{j=1}^n \theta_j \Cov(Y_{1i}, Y_{1j}\vert Y_t)\right]^2\right)\nonumber \\ 
\leq & 20 C_p\sup_{\xi\in [0,1]} \sum_{k=1}^n \E \Bigg[\sum_{i=1}^n \left(\sum_{j=1}^n \theta_j \Cov(Y_{1i}, Y_{1j}\vert Y_t=y_t) \right) \nonumber \\
&\qquad\qquad\qquad\qquad\quad \times \left(\sum_{l=1}^n \theta_l\partial_k \Cov(Y_{1i}, Y_{1l} \vert Y_t=y_t)  \right)\Bigg\vert_{y_t=t(Y_1^{\{k,-\}}+2\xi e_k )+\sqrt{t(1-t)}Z} \Bigg]^2,
}
where $\partial_k$ denotes the partial derivative with respect to the $k$th coordinate of $y_t$. We can compute the partial derivative as in the following lemma.
\begin{lemma}\label{lem:partial}
We have
\besn{\label{eq:lempartial}
&\partial_k \Cov(Y_{1i}, Y_{1l} \vert Y_t=y_t)\\
=& \frac{1}{1-t} \Cov(Y_{1i}Y_{1l}, Y_{1k} \vert Y_t=y_t)
-\frac{1}{1-t} \Cov(Y_{1i}, Y_{1k}\vert Y_t=y_t)\E (Y_{1l}\vert Y_t=y_t)\\
&-\frac{1}{1-t} \Cov(Y_{1l}, Y_{1k}\vert Y_t=y_t)\E (Y_{1i}\vert Y_t=y_t)\\
=& \frac{1}{1-t} \E \Big[(Y_{1i}-\E(Y_{1i}\vert Y_t=y_t))(Y_{1l}-\E(Y_{1l}\vert Y_t=y_t))(Y_{1k}-\E(Y_{1k}\vert Y_t=y_t))  \Big\vert Y_t=y_t  \Big].
}
\end{lemma}

\begin{proof}[Proof of \cref{lem:partial}]
Recall the distribution of $Y_1$ from \cref{eq:Isingmodel}. Write 
\be{
U(y):= \frac{1}{2}y^\top A y+h^\top y, \quad y\in \{-1,1\}^n.
}
The conditional probabilty mass function of $Y_1$ at $y\in \{-1,1\}^n$ given $Y_t=y_t$ is
\ben{\label{eq:conddens}
p(y\vert y_t)=\exp\left(U(y)-\frac{|y_t-ty|^2}{2t(1-t)} -\psi_t(y_t) \right),\ y\in \{-1,1\}^n,
}
where 
\be{
\psi_t(y_t)=\log\left[ \sum_{y\in \{-1,1\}^n}\exp\left(U(y)-\frac{|y_t-ty|^2}{2t(1-t)} \right)   \right].
}
We can compute directly that
\ben{\label{eq:partialk}
\partial_k \psi_t(y_t)=-\frac{y_{tk}}{t(1-t)}+\frac{1}{1-t}\E (Y_{1k} \vert Y_t=y_t),
}
where $y_{tk}$ denotes the $k$th coordinate of $y_t$.
Note that
\ben{\label{eq:covcomp}
\Cov(Y_{1i}, Y_{1l} \vert Y_t=y_t)=\E (Y_{1i} Y_{1l} \vert Y_t=y_t) -\E (Y_{1i} \vert Y_t=y_t)\E(Y_{1l} \vert Y_t=y_t).
}
We first compute the partial derivative with respect to the first term on the right-hand side of \cref{eq:covcomp}. 
We have, from \cref{eq:conddens},
\be{
\E (Y_{1i} Y_{1l} \vert Y_t=y_t)=\sum_{y_1\in \{-1, 1\}^n} y_{1i} y_{1l} \exp\left(U(y)-\frac{|y_t-ty|^2}{2t(1-t)} -\psi_t(y_t) \right) .
}
Differentiating with respect to $y_{tk}$, we obtain from \cref{eq:partialk} that
\besn{\label{eq:partialkil}
\partial_k \E (Y_{1i} Y_{1l} \vert Y_t=y_t)&=\sum_{y_1\in \{-1, 1\}^n} y_{1i} y_{1l} \exp\left(-U(y)-\frac{|y_t-ty|^2}{2t(1-t)} -\psi_t(y_t) \right) \\
&\qquad \times \left[-\frac{y_{tk}-t y_{1k}}{t(1-t)}+\frac{y_{tk}}{t(1-t)}-\frac{1}{1-t} \E(Y_{1k}\vert Y_t=y_t) \right] \\
&=\frac{1}{1-t}\left\{ \E(Y_{1i} Y_{1l} Y_{1k} \vert Y_t=y_t)- \E(Y_{1i} Y_{1l}  \vert Y_t=y_t)\E( Y_{1k} \vert Y_t=y_t) \right\}.
}
Similar calculations yield
\ben{\label{eq:partialki}
\partial_k \E (Y_{1i} \vert Y_t=y_t)=\frac{1}{1-t} \left\{\E(Y_{1i} Y_{1k} \vert Y_t=y_t)- \E(Y_{1i}   \vert Y_t=y_t)\E( Y_{1k} \vert Y_t=y_t) \right\}.
}
Combining \cref{eq:covcomp} with \cref{eq:partialkil,eq:partialki} proves the first equation of \cref{eq:lempartial}. The second equation of \cref{eq:lempartial} follows by algebra.
\end{proof}


\textbf{Step 4: Final simplifications.}
From \cref{eq:conddens}, the conditional probability mass function of $Y_1$ given $Y_t=y_t$ is
\be{
\mu(y\vert Y_t=y_t)\propto \mu(y)\exp\left(-\frac{|y-y_t/t|^2}{2(1-t)/t} \right)\propto \exp\left(\frac{1}{2} y^\top A y +(h+\frac{y_t}{1-t})^\top y   \right),\quad y\in \{-1,1\}^n.
}
This is of the same form as \cref{eq:Isingmodel}, except for a change of the external field. From \cref{eq:afterpoincare}, \cref{lem:partial}, taking supremum over all possible values of $y_t$ and then dropping the expectation on the right-hand side of \cref{eq:afterpoincare}, we obtain (recall the notation \cref{eq:Bilk})
\ben{\label{eq:step41}
\Var\left(\left|\theta^\top \Cov(Y_1\vert Y_t)\right|^2\right) 
\leq  \frac{20 C_p}{(1-t)^2}\sup_{h\in \IR^n} \sum_{k=1}^n \left\{\sum_{i=1}^n \Big[\sum_{j=1}^n \theta_j B_{ij}^{(h)}\Big]\Big[\sum_{l=1}^n \theta_l B_{ilk}^{(h)}\Big] \right\}^2.
} 
From \cref{eq:deduction}, \cref{eq:step31}, \cref{eq:step41} and
\be{
\int_0^{1-\eps} \frac{1}{(1-t)^6} dt\leq \frac{1}{5\eps^5},
} 
we obtain, for $\eps<1/2$,
\bes{
&\mathcal{W}_2 (\mathcal{L}(W_n), N(\mu_n, \sigma_n^2))\leq 8\sqrt{\eps}\\
& + \sqrt{\frac{4}{\eps^6(1-\|A\|_{\text{op}})} \sup_{h\in \IR^n} \sum_{k=1}^n \bigg[\sum_{i=1}^n \Big(\sum_{j=1}^n \theta_j B_{ij}^{(h)}\Big)\Big(\sum_{l=1}^n \theta_l B_{ilk}^{(h)}\Big) \bigg]^2 }.
}
The final bound \cref{eq:thm1} follows by optimizing $\eps$.

\section{Proof of \cref{eq:cor2}}\label{sec:3}
It suffices to prove that $B_{ij}^{(h)}$ in \cref{eq:Bilk} is exponentially small in terms of $d(i,j)$ and a similar result for $B_{ilk}^{(h)}$. Such results for fast-mixing Ising models with short-range interaction are well-known in the literature; see \cite{holley1976applications}, \cite{holley11985possible}, \cite{martinelli1994approach} and \cite{martinelli1999lectures}. For the sake of completeness, we still give the proof below.

\medskip

\textbf{Step 1: From Poincar\'e to exponential ergodicity.}
Note that the Dirichlet form $\mathcal{E}_\mu(f,f)$ in \cref{eq:Poincare} corresponds to the continuous time Glauber dynamics with generator (cf. the displayed equation below Eq.(6) in \cite{eldan2022spectral})
\be{
(\mathcal{L}_\mu f)(x)=\sum_{i=1}^n (\E_\mu [f(X)\vert X_{\sim i}=x_{\sim i}] -f(x)).
}
This process can be described as follows: Each spin in the system is associated with a Poisson clock that ticks independently of the others. The ticking rate of the clock for each spin is 1. When the clock for a spin $i$ ticks, the spin value is resampled according to its conditional distribution given all the other spin values $X_{\sim i}$.

From the Poincar\'e inequality \cref{eq:Poincare}, the continuous time process satisfies $L^2$-exponential ergodicity (cf. \cite[Theorem~1.1.1]{wang2005functional} and \cite[Theorem~2.18]{van2016probability}), that is,
\ben{\label{eq:L2ergo}
\|P_t f-\mu f\|_{L^2(\mu)}\leq e^{-t/C_p} \|f-\mu f\|_{L^2(\mu)},
}
where $C_p:=\sup_{n\geq 1} (1/(1-\|A_n\|_{\text{op}}))$ (we take it to be the supremum of the Poincar\'e constants over all $n$, which ensures a uniform exponential decay rate).

\medskip

\textbf{Step 2: From exponential ergodicity to correlation decay.}  

Fix any given index $n$.
We first bound $B_{ij}^{(h)}=\Cov(X_i, X_j)$ where the covariance is computed under the model \cref{eq:Isingmodel} with an arbitrary external field $h$. The basic idea from \cite{martinelli1999lectures} is as follows: We let $f$ in \cref{eq:L2ergo} be $f_1(x)=x_i x_j$, $f_2(x)=x_i$ or $f_3(x)=x_j$. We first start the continuous time Glauber dynamics $\{X(t)\}_{t\geq 0}$ from a suitable \emph{fixed} position $X(0)=x_0\in \IR^n$. Then, we choose $t=\delta\cdot d(i,j)$ for a sufficiently small constant $\delta>0$. Finally, we argue that $P_t f$ deviates from $\mu f$ by an exponentially small amount in terms of $t$, yet, with overwhelming probability, $X_i(t)$ does not depend on $X_j(t)$.

More precisely, for the choices of $f=f_1, f_2,$ or $f_3$ above, we have $|f(x)|\leq 1$ and the right-hand side of \cref{eq:L2ergo} is bounded by $e^{-t/C_p}$ (recall variance is bounded by range$^2/4$).
Then, there exists $x_0\in \IR^n$ (from Markov's inequality) such that
\ben{\label{eq:uniformdecay}
|(P_t f)(x_0)-\mu f|\leq 2e^{-t/C_p},\quad f=f_1, f_2,\ \text{or}\ f_3.
}
From \cref{eq:uniformdecay} and $|f(x)|\leq 1$, we have
\besn{\label{eq:comparison}
&\Big|\big[\E X_i(t) X_j (t) -\E X_i(t) \E X_j (t)\big]-\Cov(X_i, X_j)\Big|\\
=& \Big|\big[(P_t f_1)(x_0) - (P_t f_2)(x_0)\cdot (P_t f_3)(x_0)\big]-[\mu f_1 -\mu f_2 \cdot \mu f_3]\Big| \\
\leq& 6 e^{-t/C_p}. 
}
As in \cite[Eq.(3.7)]{martinelli1999lectures}, let $E(i,t,l)$ be the event that there exist a positive integer $m\geq 1$ and a collection of sites $\{i_0,\dots, i_m\}$ and times $\{t_0,\dots, t_m\}$ such that

i) $0<t_0<\dots<t_m\leq t$ and at each time $t_k$ the Poisson clock associated with $i_k$ ticks, and

ii) $d(i_0, i_m)\geq l$, $d(i_k, i_{k+1})\leq r$. 

It was shown in \cite[below Eq.(3.8)]{martinelli1999lectures} that there exists a constant $k_0$, depending only on the dimension $d$ and the interaction range $r$, such that 
\be{
P(E(i,t,l))\leq e^{-t},\quad \text{for}\ t\leq l/k_0.
}
Note that if the event $E(i,t,l)^c$ occurs, then $X_i(t)$ is independent of $X_j(t)$ if $d(i,j)>l$.
Therefore, for $d(i,j)>l$ and $t\leq l/k_0$, we have
\bes{
&\E X_i(t) X_j (t) -\E X_i(t) \E X_j (t)\\
=& \E X_i(t) X_j (t) 1_{E(i,t,l)^c} -\E X_i(t)1_{E(i,t,l)^c} \E X_j(t) 1_{E(i,t,l)^c} +O(e^{-t})\\
=&O(e^{-t}).
}
This, together with \cref{eq:comparison}, yields the desired exponential decay (in $d(i,j)$) of the covariance $\Cov(X_i, X_j)$.
By a similar argument, $B_{ilk}^{(h)}$ in \cref{eq:Bilk} is also exponentially small in $\max\{d(i,k), d(l,k)\}$.
Therefore, the right-hand side of \cref{eq:thm1} with $\theta_j=1/\sqrt{n}$, $1\leq j\leq n$, tends to 0 as $n\to \infty$.
This finishes the proof of \cref{eq:cor2}.

\section{Proof of \cref{thm:3}}\label{sec:4}

For any given index $n$, let $X=(X_1,\dots, X_n)^\top$ follow the model in \cref{eq:Isingmodel}. Both $X$ and the model parameters \cref{eq:Isingmodel} depend on $n$, but we omit this dependence in the notation for convenience. Below, it is important to note that the error bounds do not depend on $n$.

In \cite[Corollary~3.4 with $d\equiv 0$]{kunsch1982decay}, it was shown that $\max_{1\leq i\leq n}\sum_{j=1}^n |\Cov(X_i, X_j)|$ is uniformly bounded in $n$ under the Dobrushin condition (see \cite[Theorem~2]{gross1979decay} for an earlier result under stronger assumptions). It seems difficult to generalize their arguments for higher order correlations. Instead, we employ a coupling argument below, for which we need the $X_i$'s to be positively associated.

It follows from \cite[Corollary~3.4 and Remark~3.4 ii)]{kunsch1982decay} that, provided that $\beta_n$ and $\gamma_n$ below can be chosen to be smaller than 1,
\ben{\label{eq:kun}
\max_{1\leq i\leq n}\sum_{j=1}^n |\Cov(X_i, X_j)|\leq \frac{4}{(1-\beta_n)(1-\gamma_n)},
}
where, recalling $X_{\sim i}$ denotes the collection $\{X_j\}_{j\ne i}$,
\ben{\label{eq:betagamma}
\beta_n\geq \max_{1\leq j\leq n}\sum_{i=1}^n C_{ij}, \quad \gamma_n\geq \max_{1\leq i\leq n}\sum_{j=1}^n C_{ij},
}
\ben{\label{eq:Cij}
C_{ij}=\sup \{d_{\text{TV}}(\mathcal{L}(X_j \vert X_{\sim j}=x), \mathcal{L}(X_j \vert X_{\sim j}=y): x=y\  \text{except at}\  i\text{th spin} \}
}
and $d_{\text{TV}}$ denotes the total variation distance.
In fact, for the positively associated case that we consider, if $\gamma_n\leq \gamma<1$, then 
\be{
\max_{1\leq i\leq n}\sum_{j=1}^n \Cov(X_i, X_j)\leq C(\gamma),
}
where $C(\gamma)$ is a positive constant that depends only on $\gamma$ (cf. \cref{eq:4.145}). 

By a direct computation,
\be{
C_{ij}\leq \sup_{M\in \IR} \Big|\frac{1+\tanh(M+A_{ij})}{2}-\frac{1+\tanh(M-A_{ij})}{2}\Big|\leq  A_{ij}.
}
Therefore, recalling \cref{eq:alpha} and the assumption that $\sup_{n\geq 1}\alpha_n\leq \alpha<1$, we can take
\be{
\beta_n=\gamma_n=\alpha,\quad n\geq 1.
}
We remark that the usual Dobrushin's condition refers to $\beta_n$ being bounded away from 1 (see \cref{rem:beyongquadra} for a discussion for the general case beyond quadratic interactions).

For $\alpha<1$, the Poincar\'e inequality \cref{eq:rem11} holds with $C_p=1/(1-\alpha)$ (see \cite[Theorem~2.1]{wu2006poincare} for a stronger result). By \cref{rem:1.1}, \cref{thm:1} still applies.
For the right-hand side of \cref{eq:thm1} (with $\theta_j=1/\sqrt{n}, 1\leq j\leq n$) to vanish, it suffices to show
\ben{\label{eq:vanish1}
\frac{1}{n^2}\sum_{k=1}^n \Big\{\E \big[\sum_{i=1}^n \Cov(X_i, S) \tilde X_i\big]\tilde S \tilde X_k \Big\}^2\to 0,\ \text{as}\ n\to \infty,
}
where $\tilde X_k=X_k-\E X_k$, $\tilde S=\sum_{k=1}^n \tilde X_k $.
From \cref{eq:kun} and positive association, we have $0\leq \Cov(X_i, S)\leq C$. Hereafter, $C$ denotes positive constants (possibly different in each appearance) that depend only on $\alpha$, and $O(1)$ denotes a constant bounded in absolute value by $C$. 

For simplicity of presentation, we will argue for
\ben{\label{eq:ddd}
\E[\sum_{i=1}^n \tilde X_i]\tilde S \tilde X_k=\E \tilde S^2 \tilde X_k=O(1)
}
and the same proof, with minor modification (see below \cref{eq:ccc}), leads to
\ben{\label{eq:actualneed}
\E \big[\sum_{i=1}^n \Cov(X_i, S) \tilde X_i\big]\tilde S \tilde X_k=O(1).
}


We first aim to show that
\ben{\label{eq:aim}
\E ( (\tilde S^{(k)})^2\vert X_k=1) -\E ( (\tilde S^{(k)})^2\vert X_k=-1)=O(1),
}
where $\tilde S^{(k)}:=\tilde S-\tilde X_k$. In fact, the arguments below leading to \cref{eq:aim} constitute the main part of the proof.

We use coupling to bound the difference in \cref{eq:aim}. Consider two coupled discrete-time Glauber dynamics (each time they select a uniformly chosen random spin and resample its value according to the conditional distribution given all the other spin values) $Y_t^{k,+}$ and $Y_t^{k,-}$, $t=0,1,2,\dots$, starting from the same initial state and having $\mathcal{L}(X_{\sim k} \vert X_k=1)$ and $\mathcal{L}(X_{\sim k} \vert X_k=-1)$ as their stationary distributions, respectively. In each step, they update the same spin and keep the partial order $Y_t^{k,-}\preceq Y_t^{k,+}$ (this is possible because of positive association). Moreover, they are coupled in the minimal way in each step to be as close as possible.
Let $D_t$ be the number of different spins between the two Markov chains at time $t$.
We first argue that $D_t$, $t=0,1,2,\dots$, has a negative drift when it is sufficiently large, hence in the steady state, its second moment is bounded.


From the definition of $D_t$, we have
\ben{\label{eq:EDt}
D_t=\sum_{1 \leq j \leq n: j\ne k} 1\left\{Y_{t,j}^{k,+} \neq Y_{t,j}^{k,-}\right\},
}
where $1\{\cdots\}$ denotes the indicator random variable and $Y_{t,j}^{k,+}$ ($Y_{t,j}^{k,-}$, resp.) denotes the value of $Y_{t}^{k,+}$ ($Y_{t}^{k,-}$, resp.) at the $j$th spin.
From the Glauber dynamics and the coupling defined above, we have, for $j\ne k$,
\besn{\label{eq:onestep}
&P\left(Y^{k,+}_{t+1,j} \neq Y^{k,-}_{t+1,j} \vert Y_t^{k,+}, Y_t^{k,-}\right)\\
=&\left(1-\frac{1}{n-1}\right) 1\left\{Y^{k,+}_{t,j} \neq Y^{k,-}_{t,j}\right\}+\frac{1}{n-1} d_{\text{TV}}\left(\mu^{k,+}_j(\cdot \vert Y_t^{k,+}), \mu_j^{k,-}(\cdot \vert Y_t^{k,-})\right),
}
where $\mu^{k,+}_j(\cdot \vert Y_t^{k,+})$ ($\mu_j^{k,-}(\cdot \vert Y_t^{k,-})$, resp.) denotes the conditional distribution
of the $j$th spin value given $X_k=+1$ ($X_k=-1$, resp.) and other spin values equal to those in $Y^{k,+}_t$ ($Y^{k,-}$, resp.).
From the definition of $C_{ij}$ in 
\cref{eq:Cij}, we have, recalling the different values at the $k$th spin and using the triangle inequality for $d_{\text{TV}}$,
\ben{\label{eq:dtvckj}
d_{\text{TV}}\left(\mu^{k,+}_j(\cdot \vert Y_t^{k,+}), \mu_j^{k,-}(\cdot \vert Y_t^{k,-})\right) \leq C_{kj}+\sum_{i\ne k, j}C_{ij} 1\{Y^{k,+}_{t,i}\ne Y^{k,-}_{t,i}\}.
}
From \cref{eq:onestep}, \cref{eq:dtvckj} and \cref{eq:EDt}, we have
$$
\begin{aligned}
& \E\left\{\sum_{1 \leq j \leq n: j\ne k} 1\left\{Y_{t+1,j}^{k,+} \neq Y_{t+1,j}^{k,-}\right\} \vert Y_t^{k,+}, Y_t^{k,-}\right\} \\
= & \sum_{1 \leq j \leq n: j\ne k}\left[\left(1-\frac{1}{n-1}\right) 1\left\{Y^{k,+}_{t,j} \neq Y^{k,-}_{t,j}\right\}+\frac{1}{n-1} d_{\text{TV}}\left(\mu^{k,+}_j(\cdot \vert Y_t^{k,+}), \mu_j^{k,-}(\cdot \vert Y_t^{k,-})\right)\right] \\
\leq & \sum_{1 \leq j \leq n: j\ne k}\left[\left(1-\frac{1}{n-1}\right) 1\left\{Y^{k,+}_{t,j} \neq Y^{k,-}_{t,j}\right\}+\frac{1}{n-1} \left[C_{kj}+\sum_{i\ne k, j}C_{ij} 1\{Y^{k,+}_{t,i}\ne Y^{k,-}_{t,i}\} \right]\right] \\
\leq & \left(1-\frac{1}{n-1}\right)  D_t+\frac{1}{n-1}\left[\alpha+ \sum_{i\ne k} 1\left\{Y^{k,+}_{t,i}\ne Y^{k,-}_{t,i}\right\} \alpha\right] \\
= & \left(1-\frac{1-\alpha}{n-1}\right) D_t+\frac{\alpha}{n-1}.
\end{aligned}
$$
Therefore,
\ben{\label{eq:ED-D}
\E(D_{t+1}-D_t\vert D_t)\leq -\frac{1-\alpha}{n-1} D_t +\frac{\alpha}{n-1}.
}
Similarly, using 
\bes{
&P(D_{t+1}-D_t=1\vert Y_t^{k,+}, Y_t^{k,-})\\
\leq &\sum_{1 \leq j \leq n: j\ne k} P\left(Y^{k,+}_{t+1,j} \neq Y^{k,-}_{t+1,j}, Y^{k,+}_{t,j} = Y^{k,-}_{t,j} \vert Y_t^{k,+}, Y_t^{k,-}\right)\\
=&\sum_{1 \leq j \leq n: j\ne k}\frac{1}{n-1} d_{\text{TV}}\left(\mu^{k,+}_j(\cdot \vert Y_t^{k,+}), \mu_j^{k,-}(\cdot \vert Y_t^{k,-})\right),
}
we have
\ben{\label{eq:PD-D=1}
P(D_{t+1}-D_t=1\vert D_t)\leq \frac{\alpha(D_t+1)}{n-1}.
}
Note that $D_{t+1}-D_t\in \{-1,0,1\}$.
By comparing the process $D_t$, $t=0,1,2,\dots$, with a 0 reflected random walk with negative drift,
we will prove in the following lemma that the probability mass function of $D_\infty$ decays geometrically.
Hereafter, time $t=\infty$ means that the Markov chain is in the steady state. 

\begin{lemma}\label{lem:geometric}
Let $p_d$, $d=0,1,2,\dots$, denote the probability mass function of $D_\infty$. We have, for some $d_1$ depending only on $\alpha$ and any $d_2\geq 1$,
\be{
\frac{p_{d_1+d_2}}{p_{d_1}}\leq C c^{d_2},
}
where $0<c<1$ depends only on $\alpha$.
\end{lemma}
\begin{proof}[Proof of \cref{lem:geometric}]
We compute the stationary distribution using the detailed balance equation. For $i=0,1,\dots$, we have
\be{
\frac{p_{d_1+i+1}}{p_{d_1+i}}=\frac{P(D_{t+1}-D_t=1\vert D_t=d_1+i)}{P(D_{t+1}-D_t=-1\vert D_t=d_1+i+1)}\leq \frac{P_{d_1+i}(+1)}{\frac{1-\alpha}{n-1}(d_1+i+1)-\frac{\alpha}{n-1}+P_{d_1+i+1}(+1)}
}
where we used \cref{eq:ED-D} in the last inequality and denoted
\be{
P_{d_1+i}(+1):=P(D_{t+1}-D_t=1\vert D_t=d_1+i).
}
Therefore, we have
\bes{
\frac{p_{d_1+d_2}}{p_{d_1}}\leq& \prod_{i=0}^{d_2-1} \frac{P_{d_1+i}(+1)}{\frac{1-\alpha}{n-1}(d_1+i+1)-\frac{\alpha}{n-1}+P_{d_1+i+1}(+1)}\\
=&\frac{P_{d_1}(+1)}{\frac{1-\alpha}{n-1}(d_1+d_2)-\frac{\alpha}{n-1}+P_{d_1+d_2}(+1)}\prod_{i=1}^{d_2-1} \frac{P_{d_1+i}(+1) }{\frac{1-\alpha}{n-1}(d_1+i)-\frac{\alpha}{n-1}+P_{d_1+i}(+1) }.
}
From \cref{eq:PD-D=1}, we have
\bes{
\frac{p_{d_1+d_2}}{p_{d_1}}\leq& \frac{\frac{\alpha(d_1+1)}{n-1}}{\frac{1-\alpha}{n-1}d_1-\frac{\alpha}{n-1}}\prod_{i=1}^{d_2-1} \frac{\frac{\alpha(d_1+i+1)}{n-1} }{\frac{1-\alpha}{n-1}(d_1+i)-\frac{\alpha}{n-1}+\frac{\alpha(d_1+i+1)}{n-1} }\\
=&\frac{\alpha(d_1+1)}{(1-\alpha)d_1-\alpha}\prod_{i=1}^{d_2-1} \frac{\alpha(d_1+i+1)}{d_1+i}.
}
Recall $\alpha<1$. Choosing a sufficiently large $d_1$, we obtain the desired result.
\end{proof}

From \cref{lem:geometric}, we have $\E D^2_\infty\leq C$. 

To proceed with the proof, we need the following lemma.
\begin{lemma}\label{lem:meandecrease}
Let $A$ be the event that $D_\infty=d\geq 0$ and $Y_\infty^{k,+}$ and $Y_\infty^{k,-}$ differ in the $d$ given locations $i_1,\dots, i_d$ (and take the same spin values in the other $n-1-d$ locations).
Let the event $B$ be 
\be{
B=\{Y^{k,+}_{\infty,i_1}=\dots=Y^{k,+}_{\infty,i_d}=1\},
}
that is, there is no restriction for the other spin values. 
Then, we have
\ben{\label{eq:meandecrease}
\E (\sum_{j\notin\{k,i_i,\dots, i_d\}}Y_{\infty,j}^{k,+}\vert A)\leq \E (\sum_{j\notin\{k,i_i,\dots, i_d\}}Y_{\infty,j}^{k,+}\vert B).
}
In fact, restricting the two Markov chains $Y_\infty^{k,+}$ and $Y_\infty^{k,-}$ (coupled to keep the partial order as above \cref{eq:EDt}) to take the same values for the $n-1-d$ locations decreases (increases, resp.) the expected spin values of the larger (smaller, resp.) Markov chain $Y_\infty^{k,+}$ ($Y_\infty^{k,-}$, resp.) in these locations.
\end{lemma}
\begin{proof}[Proof of \cref{lem:meandecrease}]
We first create a discrete-time Markov chain on spin values at locations $\{1,\dots, n\}\backslash \{k, i_1, \dots, i_d\}$ which has the stationary distribution as the conditional distribution of $Y_\infty^{k,+}$ given the event $A$.
Recall that the two Markov chains $Y^{k,+}_t$ and $Y^{k,-}_t$ defined above \cref{eq:EDt} are coupled in the minimal way to be as close as possible.
When updating the spin value at location $j$, we compute 
\be{
p_1:=P(Y_{\infty,j}^{k,+}=1\vert Y^{k,+}_{\infty, \sim j}=y^+)
}
and
\be{
p_2:=P(Y_{\infty,j}^{k,-}=1\vert Y^{k,-}_{\infty, \sim j}=y^-),
}
where $\sim j$ means spin values except at location $j$ and $y^+$ and $y^-$ differ exactly at locations $\{i_1,\dots, i_d\}$. Recalling $y^-\preceq y^+$, we have $p_2\leq p_1$. We update the spin value of the new Markov chain at location $j$ according to 
\be{
P(+1)=\frac{p_2}{1-(p_1-p_2)}, \quad P(-1)=\frac{1-p_1}{1-(p_1-p_2)}.
}
We update a uniformly chosen random spin each time.
To argue that the resulting Markov chain indeed has the desired stationary distribution (the conditional distribution of $Y^{k,+}_\infty$ given the event $A$, i.e., $Y^{k,+}_\infty$ and $Y^{k,-}_\infty$ have to take the same value at any spin $j\notin \{i_1,\dots, i_d, k\}$), we use the fact that the stationary distribution of a reversible Markov chain $\{Y^{k,+}_t, Y^{k,-}_t\}_{t=0,1,\dots}$, when restricted to two connected states $\{y^+\cup\{+1\}, y^-\cup\{+1\}\}$ and $\{y^+\cup\{-1\}, y^-\cup\{-1\}\}$, is proportional to the one-step transition probabilities between these two states (note that $P(+)$ and $P(-)$ above are exactly the transition probabilities between the two states in the minimal coupling above \cref{eq:EDt} to keep $Y^{k,-}_t$ and $Y^{k,+}_t$ as close as possible).

If we had used $p_1$ as the probability of updating the $j$th spin to the value 1, we would have obtained a Markov chain having the stationary distribution as the conditional distribution of $Y_\infty^{k,+}$ given the event $B$.

Because of $p_1\geq P(+1)$ and positive association, we can couple the above two Markov chains so that they start from the same initial condition and the second chain always takes larger values in the partial order of spin configurations. Since both the Markov chains converge to their stationary distributions, this proves the lemma.








\end{proof}


Recall $Y^{k,+}_{t}$ and $Y^{k,-}_{t}$ as defined above \cref{eq:EDt} and $t=\infty$ denotes the steady state.
Let
\be{
\tilde S^{k,+}_\infty=\sum_{1\leq j\leq n: j\ne k} Y^{k,+}_{\infty, j},\quad \tilde S^{k,-}_\infty=\sum_{1\leq j\leq n: j\ne k} Y^{k,-}_{\infty, j}
}
Given the event $B$ (i.e., fixing $d$ additional spins $\{i_1,\dots, i_d\}$ to take value $+1$ besides the $k$th spin), 
we will prove the following lemma.

\begin{lemma}\label{lem:Cd1}
We have
\ben{\label{eq:exptgivenB}
\E (\tilde S^{k,+}_\infty \vert B)\leq C(d+1).
}
\end{lemma}
\begin{proof}[Proof of \cref{lem:Cd1}]
We consider the two coupled discrete-time Glauber dynamics  $Y_t^{k,+}$ and $Y_t^{k,-}$ as in the paragraph above Eq.\cref{eq:EDt}, except that now we always update the spins $\{i_1,\dots, i_d\}$ to $+1$ for $Y_t^{k,+}$ and to $-1$ for $Y_t^{k,-}$. 
We still denote the number of spins with different values for $Y_t^{k,+}$ and $Y_t^{k,-}$ by $D_t$. Note that $D_t\geq d$.
Following similar arguments leading to \cref{eq:ED-D,eq:PD-D=1}, we obtain
\be{
\E(D_{t+1}-D_t\vert D_t)\leq -\frac{1-\alpha}{n-1} D_t +\frac{\alpha+d}{n-1}
}
and 
\be{
P(D_{t+1}-D_t=1\vert D_t)\leq \frac{\alpha(D_t+1)}{n-1}.
}
Following similar arguments leading to \cref{lem:geometric}, the probability mass function of $D_t$ decays geometrically when $D_\infty\geq C(d+1)$ for a sufficiently large $C$. In particular,
\be{
\E D_\infty\leq C(d+1).
}
Squeezing a third Glauber dynamics (which has the stationary distribution $\mu$ and also updates the $k$th spin in each time step) in between $Y_t^{k,+}$ and $Y_t^{k,-}$, we obtain 
\be{
\E (\tilde S^{k,+}_\infty \vert B)\leq 2\E D_\infty\leq C(d+1).
}
\end{proof}

From \cref{eq:meandecrease}, \cref{eq:exptgivenB} and $\E D_\infty^2\leq C$, the left-hand side of \cref{eq:aim} can be upper bounded as
\besn{\label{eq:argumentbegin}
&\E ( (\tilde S^{(k)})^2\vert X_k=1) -\E ( (\tilde S^{(k)})^2\vert X_k=-1)\\
=&\E (\tilde S_\infty^{k,+})^2 -\E (\tilde S_\infty^{k,-})^2\\
=&4\E \tilde S_\infty^{k,+} D_\infty-4\E D_\infty^2\\
\leq & C\E(D_\infty+1)^2\leq C.
}
By a similar argument,
\bes{
&\E (\tilde S_\infty^{k,+})^2 -\E (\tilde S_\infty^{k,-})^2\\
=&4\E \tilde S_\infty^{k,-}  D_\infty +4\E D_\infty^2\\
\geq& -C.
}
Combining the above upper and lower bounds proves \cref{eq:aim}.

From \cref{eq:aim}, we have
\besn{\label{eq:bbb}
\E(\tilde S^{(k)})^2 \tilde X_k
=P(X_k=1) (1-\E X_k) \big\{\E ( (\tilde S^{(k)})^2\vert X_k=1) -\E ( (\tilde S^{(k)})^2\vert X_k=-1) \big\}=O(1).
}
Because 
\ben{\label{eq:4.145}
|\E(\tilde S^{(k)}\vert \tilde X_k)|\leq \E(\tilde S^{k,+}_\infty-\tilde S^{k,-}_\infty )=2\E D_\infty=O(1),
}
we have
\ben{\label{eq:ccc}
\E \tilde S \tilde X_k^2=\E \tilde X_k^3+ \E \tilde S^{(k)} \tilde X_k^2 =O(1).
}
Combining \cref{eq:bbb,eq:ccc}, we obtain \cref{eq:ddd}.

By a straightforward modification of the arguments from \cref{eq:argumentbegin} to \cref{eq:ccc} using the boundedness and positivity of $\Cov(X_i,S)$, we obtain \cref{eq:actualneed}.
Therefore, we obtain \cref{eq:vanish1} and 
finish the proof of \cref{eq:cor3}.


Note that
\be{
P(X_i=+1\vert X_{\sim i})=\frac{\exp(h_i+\sum_{j: j\ne i}A_{ij}X_j)}{\exp(h_i+\sum_{j: j\ne i}A_{ij}X_j)+\exp(-h_i-\sum_{j: j\ne i}A_{ij}X_j)}.
}
Therefore, if $\|h\|_\infty$ is bounded by a universal constant, then, together with the Dobrushin condition, we have that $\inf_{i} \Var(X_i)$ is bounded away from 0. Together with positive association, we have that $\sigma_n^2$ is bounded away from 0. Therefore, \cref{eq:cor32} follows from \cref{eq:cor3}. This finishes the proof of \cref{thm:3}.

\section{Proof of \cref{thm:4}}\label{sec:5}

From the discussion above \cref{thm:4}, with probability tending to 1, the Poincar\'e inequality \cref{eq:rem11} with constant $C_p=c_\beta$ depending only on $\beta$ holds for the SK model with sufficiently small $\beta>0$. Therefore, we can apply \cref{thm:1} (with \cref{rem:1.1}). To prove \cref{eq:cor4}, it suffices to show that for the SK model with sufficiently small $\beta>0$ and arbitrary external field,
\ben{\label{eq:SK-1}
\frac{1}{n^2}\E\sum_{i=1}^n\left( \sum_{j=1}^{n} \left[ \sum_{l=1}^{n} m_{jl} \right] \left[ \sum_{k=1}^{n} m_{ijk} \right] \right)^2
\to 0,
}
where $m_{jl}$ and $m_{ijk}$ denote the two-point and three-point functions, respectively, and the expectation is with respect to the random interaction matrix.
Indeed, \cref{eq:thm1} with $1/(1-\|A\|_{\text{op}})$ replaced by $c_\beta\vee 1$ and $\theta_j=1/\sqrt{n}, 1\leq j\leq n$, together with \cref{eq:SK-1}, implies that (on the event that the Poincar\'e inequality holds, which has probability tending to 1)
\be{
\E  \mathcal{W}_2 \big(\mathcal{L}(W_n \vert H_n), N(\mu_n(H_n),\sigma_n^2(H_n))\big)\to 0,
}
hence \cref{eq:cor4} follows.
In fact, in the following, we will prove that for any $i\in \{1,\dots, n\}$,
\ben{\label{eq:SK-2}
\E\left( \sum_{j=1}^{n} \left[ \sum_{l=1}^{n} m_{jl} \right] \left[ \sum_{k=1}^{n} m_{ijk} \right] \right)^2=O(1),
}
which suffices for \cref{eq:SK-1}. Hereafter, the implicit constants in the big $O$ notation depend only on $\beta$.

To prove \cref{eq:SK-2}, we will use the dynamical version of the cavity method by \cite{adhikari2021dynamical}. In that paper, it was shown that (see \cref{eq:lemma3.1adhi} below) for any $1\leq i\ne j\leq n$,
\be{
m_{ij}=O_p(\frac{1}{\sqrt{n}}).
}
We begin with the crucial observation that because of significant cancellations between ferromagnetic and anti-ferromagnetic couplings, for any $i\in \{1,\dots, n\}$ (see \cref{two-point})
\be{
\sum_{j=1}^n m_{ij}=O_p(1).
}
A sequence of such observations (Lemmas \ref{two-point}-\ref{lemma-coeffi-A}) will eventually lead to \cref{eq:SK-2}.

The remainder of the proof is organized as follows. In \cref{subsec:SK1}, we introduce the notation, quote preparatory results from \cite{adhikari2021dynamical}, and state the necessary lemmas. In \cref{subsec:SK2}, we prove \cref{eq:SK-2}. In \cref{subsec:SK3}, we prove the results on $\E(\sigma_n^2(H_n))$ and $\Var(\sigma_n^2(H_n))$ in \cref{thm:4}. In \cref{subsec:SK4}, we give the proofs of the lemmas. 

\subsection{Preliminaries and lemmas}\label{subsec:SK1}

\paragraph{Restatement of the model.}
In this proof, we make slight changes to our notation to be consistent with the main reference \cite{adhikari2021dynamical}. We denote the spins by $\sigma=(\sigma_1,\dots, \sigma_n)^\top$, the inverse temperature by $\sqrt{t}$, and rewrite the SK model as
\be{
\mu(\sigma)\propto e^{H_n(\sigma)}
}
with the Hamiltonian
\be{
H_n(\sigma)=\sum_{1\leq i<j\leq n}g_{ij} \sigma_i \sigma_j +\sum_{i=1}^n h_i \sigma_i,
}
where $\{g_{ij}, 1\leq i<j\leq n  \}$ are i.i.d.\ Gaussians of variance $t/n$. We also set $g_{ji}=g_{ij}$ and $g_{ii}=0$ for all $i, j\in \{1,\dots, n\}$. 
For a function $f$ (observable) of the spin configuration, we denote by 
\be{
\langle f\rangle=\frac{1}{Z_n} \sum_{\sigma \in\{-1,1\}^n} f(\sigma) e^{H_n(\sigma)}, \quad Z_N=\sum_{\sigma \in\{-1,1\}^n} e^{H_n(\sigma)}
}
the Gibbs expectation (given the random interactions $\{g_{ij}, 1\leq i<j\leq n  \}$). The expectation with respect to $\{g_{ij}, 1\leq i<j\leq n  \}$ is denoted by $\E$.

\paragraph{Notation.}
As in \cite{adhikari2021dynamical}, we will need to consider expectations of observables conditionally on a sub-configuration. We will use the following notation as in \cite[Section~2]{adhikari2021dynamical}.
Let $A= \left\{j_1, j_2, \ldots, j_k\right\} \subset\{1, \ldots, n\}$, let $B \subset\{1, \ldots, n\}$ be disjoint from $A$ with $|B|=l$ and let $\tau=\left(\tau_{j_1}, \ldots, \tau_{j_k}\right) \in\{-1,1\}^k$ be a fixed $k$-particle configuration. Then, we define the reduced Hamiltonian $H_n^{[A, B]} \equiv H_{n,\left(\tau_{j_1}, \ldots, \tau_{j_k}\right)}^{[A, B]}$
by
$$
H_n^{[A, B]}(\sigma)=H_n^{[A, B]}\left(\sigma_i, i\notin A \cup B\right)=\sum_{\substack{1 \leq i<j \leq n: \\ i, j \notin A \cup B}} g_{i j} \sigma_i \sigma_j+\sum_{\substack{1 \leq i \leq n: \\ i \notin A \cup B}}\left(h+\sum_{j \in A} g_{i j} \tau_j\right) \sigma_i .
$$
$H_n^{[A, B]}(\sigma)$ plays the role of the energy of the system, conditionally on the spins $\sigma_j$ for $j \in A$ such that $\sigma_j=\tau_j$ and after the particles $\sigma_j$ for $j \in B$ have been removed from the system (or equivalently, setting $\sigma_j=0$ for $j\in B$). For disjoint subsets $A, B \subset\{1, \ldots, n\}$, we then denote by $\langle\cdot\rangle^{[A, B]}$ the Gibbs measure induced by the reduced Hamiltonian $H_n^{[A, B]}$. We abbreviate $\langle\cdot\rangle^{[A]} \equiv\langle\cdot\rangle^{[A, \emptyset]},\langle\cdot\rangle^{(B)} \equiv\langle\cdot\rangle^{[\emptyset, B]}$ as well as $\langle\cdot\rangle \equiv\langle\cdot\rangle^{[\emptyset, \emptyset]}$. In particular, $\langle\cdot\rangle$ denotes the usual Gibbs measure induced by $H_n=H_n^{[\emptyset, \emptyset]}$. By slight abuse of notation, if $A=\{i\}$ is a set of only one element, we write for simplicity
$$
\langle\cdot\rangle^{[i]}:=\langle\cdot\rangle^{[\{i\}]}, \quad\langle\cdot\rangle^{(i)}:=\langle\cdot\rangle^{(\{i\})} .
$$

For an observable $f$, notice that $\langle f\rangle^{[A]}$ is equal to the conditional expectation of $f$, given the spins $\sigma_j$ for $j \in A$. Observables of particular interest will be the magnetizations $m_i^{[A]}$,
the two-point functions $m_{i j}^{[A]}$ and the three-point functions $m_{i j k}^{[A]}$, defined by

$$
\begin{aligned}
& m_i^{[A]}=\left\langle\sigma_i\right\rangle^{[A]}, \quad m_{i j}^{[A]}=\left\langle\sigma_i \sigma_j\right\rangle^{[A]}-\left\langle\sigma_i\right\rangle^{[A]}\left\langle\sigma_j\right\rangle^{[A]}, \\
& m_{i j k}^{[A]}=\left\langle\left(\sigma_i-\left\langle\sigma_i\right\rangle^{[A]}\right)\left(\sigma_j-\left\langle\sigma_j\right\rangle^{[A]}\right)\left(\sigma_k-\left\langle\sigma_k\right\rangle^{[A]}\right)\right\rangle^{[A]}.
\end{aligned}
$$
They are effectively mixed cumulants of the spins under the conditional Gibbs measure.
If $i\in A$ (hence $\sigma_i$ is fixed), we see from the definitions that $m_{ij}^{[A]}=m_{ijk}^{[A]}=0$. 
If $A=\emptyset$, we simply write $m_i, m_{i j}$ and $m_{i j k}$, respectively.


Given disjoint subsets $A, B \subset\{1, \ldots, n\}$, an index $i \in A$ and an observable $f$, we introduce furthermore the notation
\ban{
&\delta_i\langle f\rangle^{[A, B]}=\frac{1}{2} \sum_{\sigma_i= \pm 1} \sigma_i\langle f\rangle^{[A, B]}\left(\sigma_i\right),\label{eq:defdelta}\\
&\varepsilon_i\langle f\rangle^{[A, B]}=\frac{1}{2} \sum_{\sigma_i= \pm 1}\langle f\rangle^{[A, B]}\left(\sigma_i\right), \label{eq:defeps}\\
&\Delta_i\langle f\rangle^{[A, B]}
=\varepsilon_i\langle f\rangle^{[A, B]}
-\langle f\rangle^{[A\backslash \{i\}, B\cup \{i\}]}.\label{eq:defDelta}
}

The above operations are crucial in the following cavity method (recursive argument).
In particular, $\delta_i\langle f\rangle^{[A, B]}$, $\varepsilon_i\langle f\rangle^{[A, B]}$ and $\Delta_i\langle f\rangle^{[A, B]}$ are all functions of $\{\sigma_j, j \in A \backslash\{i\}\}$.
$\varepsilon_i\langle f\rangle^{[A, B]}$ ($\delta_i\langle f\rangle^{[A, B]}$, resp.) is defined by taking the average (signed average, resp.) over the possible values of $\sigma_i=\pm 1$. $\Delta_i\langle f\rangle^{[A, B]}$ is defined by taking the difference between the average value and the Gibbs expectation with $i$ removed.

Finally, we denote by $C$ generic constants that may vary from line to line and that are independent of all parameters, unless specified otherwise. If a constant depends on parameters, say $t$ and $\epsilon$, we denote this typically by subscripts, i.e. $C_{t, \epsilon}$.

\paragraph{Preliminaries.}
We quote the following useful representations and bounds on the two-point and three-point functions from \cite{adhikari2021dynamical}.

From \cite[Eq.(3.1) and Eq.(3.2)]{adhikari2021dynamical}, we have
\ben{\label{eq:mij}
m_{i j}^{[A]}=\left[1-\left(m_i^{[A]}\right)^2\right] \delta_i m_j^{[A \cup\{i\}]}
}
and 
\ben{\label{eq:mijk}
m_{i j k}^{[A]}=\left[1-\left(m_i^{[A]}\right)^2\right] \delta_i m_{j k}^{[A \cup\{i\}]}-2 m_i^{[A]} m_{i k}^{[A]} \delta_i m_j^{[A \cup\{i\}]},
}
for any $A\subset \{1,\dots, n\}$ and $i,j,k\in \{1,\dots, n\}$. Note that although  the equations are only stated for $i,j,k\notin A$ in the reference, they in fact hold for any $i,j,k$ because if any of the indices $i,j,k$ is in $A$ then both sides of the equation become 0.

As in \cite{adhikari2021dynamical}, we can view the $\left(g_{i k}\right)_{k \notin A}=\left(g_{i k}(t)\right)_{k \notin A}$ in $H_n^{[A \cup\{i\}]}$ as Brownian motions at time $t$ and speed $1 / n$ to rewrite the difference $\delta_i m_j^{[A \cup\{i\}]}$ in \cref{eq:mij} through It\^o\!\!'s lemma and the identity
\ben{\label{eq:mkkj}
m_{kkj}^{[A\cup \{i\}]}=-2m_k^{[A\cup \{i\}]} m_{kj}^{[A\cup \{i\}]}
}
as (see \cite[Eq.(3.3)]{adhikari2021dynamical})
\ben{\label{eq:SI1}
\delta_i m_j^{[A \cup\{i\}]}=\sum_{k \notin A} \int_0^t \varepsilon_i m_{k j}^{[A \cup\{i\}]}(s) d g_{i k}(s)-\sum_{k \notin A} \int_0^t \delta_i\left(m_k^{[A \cup\{i\}]} m_{k j}^{[A \cup\{i\}]}\right)(s) \frac{d s}{n},\quad \text{for}\  j\ne i.
}
Here and throughout this paper, we abbreviate $\langle f\rangle^{[A \cup\{i\}]}(s)=\langle f\rangle^{[A \cup\{i\}]}\left(\left(g_{i l}(s)\right)_{l \notin A}\right)$ for any observable $f$.
Moreover, we often write $\langle f\rangle^{[A \cup\{i\}]}(t)$ simply as $\langle f\rangle^{[A \cup\{i\}]}$, which is consistent with the earlier notation before introducing the Brownian motion.
Differentiating the identity \cref{eq:mij} with respect to the external field in the direction of $\sigma_k$, we find that (see \cite[below Eq.(3.5)]{adhikari2021dynamical})
\besn{\label{eq:SI2}
\delta_i m_{j k}^{[A \cup\{i\}]}= & \sum_{l \notin A} \int_0^t \varepsilon_i m_{j k l}^{[A \cup\{i\}]}(s) d g_{i l}(s)-\sum_{l \notin A} \int_0^t \delta_i\left(m_{k l}^{[A \cup\{i\}]} m_{j l}^{[A \cup\{i\}]}\right)(s) \frac{d s}{n} \\
& -\sum_{l \notin A} \int_0^t \delta_i\left(m_l^{[A \cup\{i\}]} m_{j k l}^{[A \cup\{i\}]}\right)(s) \frac{d s}{n}.
}
It was proved in \cite[lemma~3.1]{adhikari2021dynamical} that
for an arbitrary $\epsilon>0$, there exist constants $C_{t,\epsilon}, c_\epsilon > 0$ such that
\ben{\label{eq:lemma3.1adhi}
\E |m_{ij}|^{2+\epsilon}\leq \frac{C_{t,\epsilon}}{n^{1+\epsilon/2}}
}
for all \(i\ne j\in \{1,\ldots,n\}\) and \(0 \le t < c_\epsilon\).\footnote{They stated the result for $0\leq t<\log 2$ and sufficiently small $\epsilon>0$, but the version for fixed $\epsilon$ and sufficiently small $t$ can be easily deduced from their proof. Also, they stated the result for the case $h_i=h$ for all $i$, but the proof does not require that.}

\paragraph{Lemmas.} We are now ready to state the necessary lemmas. We will only need $\epsilon$ up to 6 when applying these lemmas. As can be seen from the proofs, and as also mentioned in \cite[Page 11, Remark(2)]{adhikari2021dynamical}, the constants $C_{t,\eps}$ in \cref{eq:lemma3.1adhi} and in the following lemmas can be chosen to be non-decreasing in the parameter $t$ of the SK model.


\begin{lemma}\label{two-point}
Let \(A \subset \{1,\ldots,n\}\) and choose arbitrary \(\epsilon > 0\).
Then, for some \(C_{t,\epsilon}, c_\epsilon > 0\), independent of \(n\) and \(A\), we have
\begin{equation*}
\sup_{\sigma \in \{-1,1\}^{A}}
\E\Big|\sum_{j=1}^nm_{ij}^{[A]}\Big|^{2+\epsilon}
\le C_{t,\epsilon}
\end{equation*}
for all \(i\in \{1,\ldots,n\}\) and \(0 \le t < c_\epsilon\).
Hereafter, $\sup_{\sigma \in \{-1,1\}^{A}}$ is over all configurations of the spins in $A$.
\end{lemma}

\begin{lemma}\label{three-onesum}
Let \(A \subset \{1,\ldots,n\}\) and choose arbitrary \(\epsilon > 0\).
Then, for some \(C_{t,\epsilon}, c_\epsilon > 0\), independent of \(n\) and \(A\), we have
\begin{equation*}
\sup_{\sigma \in \{-1,1\}^{A}}
\E\Big|\sum_{j=1}^nm_{ijk}^{[A]}\Big|^{2+\epsilon}
\le \frac{C_{t,\epsilon}}{n^{1+\epsilon/2}}
\end{equation*}
for all $i\ne k$ and \(0 \le t < c_\epsilon\).
\end{lemma}

We note that in \cref{three-onesum}, the assumption $i\ne k$ is necessary. Otherwise, the left-hand side is of order $O(1)$.
    
\begin{lemma}\label{three-point}
Let \(A \subset \{1,\ldots,n\}\) and choose arbitrary \(\epsilon > 0\).
Then, for some \(C_{t,\epsilon}, c_\epsilon > 0\), independent of \(n\) and \(A\), we have
\begin{equation*}
\sup_{\sigma \in \{-1,1\}^{A}}
\E\Big|\sum_{k,j=1}^nm_{ijk}^{[A]}\Big|^{2+\epsilon}
\le C_{t,\epsilon}
\end{equation*}
for all \(i\in \{1,\ldots,n\}\) and \(0 \le t < c_\epsilon\).
\end{lemma}

\begin{lemma}\label{Delta-error}
Let \(A \subset \{1,\ldots,n\}\) and choose arbitrary \(\epsilon > 0\).
Then, for some \(C_{t,\epsilon}, c_\epsilon > 0\), independent of \(n\), we have
\begin{equation*}
\sup_{\sigma \in \{-1,1\}^{A\backslash \{i\}}}
\E\Big|\sum_{j=1}^{n}\Delta_i m_{jk}^{[A]}\Big|^{2+\epsilon}\leq \frac{C_{t,\epsilon}}{n^{1+\epsilon/2}}
\end{equation*}
for all \(i\in A\), \(k\in \{1,\ldots,n\}\) and \(0 \le t < c_\epsilon\).
\end{lemma}

\begin{lemma}\label{Delta-error-higher-order}
Let \(A \subset \{1,\ldots,n\}\) and choose arbitrary \(\epsilon > 0\).
Denote $a_{j}^{(A)}=\sum_{l=1}^{n} m_{jl}^{(A)}$. Then, for some \(C_{t,\epsilon}, c_\epsilon > 0\), independent of \(n\), we have
    \begin{equation*}
        \E\Big|(a_{j}^{(A)}-a_{j}^{(A\cup \{i\})})(t)\Big|^{2+\ep}\leq \frac{C_{t,\ep}}{n^{1+\ep/2}}
    \end{equation*}
for all \(i\in \{1,\ldots,n\}\), $j\ne i$ and \(0 \le t < c_\epsilon\).
\end{lemma}

\begin{lemma}\label{two-point-coef}
Let \(A \subset \{1,\ldots,n\}\) and choose arbitrary \(\epsilon > 0\). Then, for some \(C_{t,\epsilon}, c_\epsilon > 0\), independent of \(n\), we have
    \begin{align*}
        \sup_{\sigma \in \{-1,1\}^{A}}\E\Big|\sum_{j=1}^n \left[\sum_{l=1}^{n} m_{jl}^{(A)}\right] \delta_i m_j^{[A]}(t)\Big|^{2+\ep}\le C_{t,\ep}
    \end{align*}
for all \(i\in A\) and \(0 \le t < c_\epsilon\).
\end{lemma}

\begin{lemma}\label{lemma-coeffi-A}
For sufficiently small $t>0$, we have
    \begin{equation}\label{eq-coef-three-A}
    \sup_{\sigma \in \{-1,1\}^{A}}\E\big(\sum_{j=1}^{n} \left[ \sum_{l=1}^{n} m_{jl}^{(A)} \right]\delta_i \left[\sum_{k=1}^{n}m_{jk}^{[A]}\right]\big)^2= O(1)
    \end{equation}
for all \(i\in A \subset \{1,\ldots,n\}\).
\end{lemma}

\subsection{Proof of \cref{eq:SK-2}}\label{subsec:SK2}
Recall from \cref{eq:SK-2} that we aim to show that for any $i\in \{1,\dots, n\}$ and sufficiently small $t>0$,
\ben{\label{eq:SK-2repeat}
\E\left( \sum_{j=1}^{n} \left[ \sum_{l=1}^{n} m_{jl} \right] \left[ \sum_{k=1}^{n} m_{ijk} \right] \right)^2=O(1).
}
In the following, we fix $i$.
We rewrite
\begin{equation}\label{eq:rewrite2terms}
    \begin{split}
        \sum_{j=1}^{n} \left[ \sum_{l=1}^{n} m_{jl} \right] \left[ \sum_{k=1}^{n} m_{ijk} \right] &=\sum_{j=1}^{n} \left[ \sum_{l=1}^{n} m_{jl}^{(i)} \right] \left[ \sum_{k=1}^{n} m_{ijk} \right] +\sum_{j=1}^{n} \left[ \sum_{l=1}^{n} (m_{jl}-m_{jl}^{(i)}) \right] \left[ \sum_{k=1}^{n} m_{ijk} \right]. 
    \end{split}
\end{equation}
For the first term on the right-hand side of \cref{eq:rewrite2terms}, we rewrite it using \cref{eq:mijk} as
\bes{
    &\sum_{j=1}^{n} \left[ \sum_{l=1}^{n} m_{jl}^{(i)} \right] \left[ \sum_{k=1}^{n} m_{ijk} \right] \\
    =&\left[1-\left(m_i\right)^2\right]\bigl(\sum_{j=1}^{n} \left[ \sum_{l=1}^{n} m_{jl}^{(i)} \right]\delta_i\left[\sum_{k=1}^{n}m_{jk}^{[i]}\right]\bigr)-2\,m_i\left[\sum_{k=1}^{n}m_{ik}\right]\,\bigl( \sum_{j=1}^{n}  \left[ \sum_{l=1}^{n} m_{jl}^{(i)} \right]\delta_i m_j^{[i]}\bigr).
}
From $|m_i|\leq 1$, we have
\begin{equation*}
\begin{split}
    &\E\left(\sum_{j=1}^{n} \left[ \sum_{l=1}^{n} m_{jl}^{(i)} \right] \left[ \sum_{k=1}^{n} m_{ijk} \right] \right)^2\\
    \le& 2\E \left(\sum_{j=1}^{n} \left[ \sum_{l=1}^{n} m_{jl}^{(i)} \right]\delta_i\left[\sum_{k=1}^{n}m_{jk}^{[i]}\right]\right)^2+8\E\left(\left[\sum_{k=1}^{n}m_{ik}\right]\,\bigl( \sum_{j=1}^{n}  \left[ \sum_{l=1}^{n} m_{jl}^{(i)} \right]\delta_i m_j^{[i]}\bigr)\right)^2\\
\le& 2\E \left(\sum_{j=1}^{n} \left[ \sum_{l=1}^{n} m_{jl}^{(i)} \right]\delta_i\left[\sum_{k=1}^{n}m_{jk}^{[i]}\right]\right)^2+8\sqrt{\E\left[\sum_{k=1}^{n}m_{ik}\right]^4}\cdot\sqrt{\E\bigl( \sum_{j=1}^{n}  \left[ \sum_{l=1}^{n} m_{jl}^{(i)} \right]\delta_i m_j^{[i]}\bigr)^4}.
\end{split}
\end{equation*}
From \cref{two-point}, \cref{two-point-coef} and \cref{lemma-coeffi-A}, we conclude that
\begin{equation}\label{eq:2terms1}
    \E \left(\sum_{j=1}^{n} \left[ \sum_{l=1}^{n} m_{jl}^{(i)} \right] \left[ \sum_{k=1}^{n} m_{ijk} \right] \right)^2= O(1).
\end{equation}
For the second term on the right-hand side of \cref{eq:rewrite2terms}, we will use the identity (to be proved half an page below)
    \begin{equation}\label{error}
m_{jk} - m_{jk}^{(i)}
=
(\delta_i m_j^{[i]})\, m_{ik}
+ \Delta_i m_{jk}^{[i]}
+ m_i\, (\delta_i m_{jk}^{[i]}).
\end{equation}
From \cref{error}, we have, using $\|X\|_p$ to denote $(\E |X|^p)^{1/p}$ for $p\geq 1$,
\begin{equation*}
\begin{split}
\|\sum_{j=1}^{n}(m_{jk}-m_{jk}^{(i)})\|_4 &=\|(\sum_{j=1}^{n}\delta_i m_j^{[i]})\, m_{ik}
+ \sum_{j=1}^{n}\Delta_i m_{jk}^{[i]}
+ m_i\, (\sum_{j=1}^{n}\delta_i m_{jk}^{[i]})\|_4\\
&\leq \|(\sum_{j=1}^{n}\delta_i m_j^{[i]}) m_{ik}\|_4+ \|\sum_{j=1}^{n}\Delta_i m_{jk}^{[i]}\|_4+\|m_i\, (\sum_{j=1}^{n}\delta_i m_{jk}^{[i]})\|_4.
\end{split}
\end{equation*}
For the first term in the above upper bound, we have
\begin{equation*}
       \|(\sum_{j=1}^{n}\delta_i m_j^{[i]}) m_{ik}\|_4^4=\E\bigl((\sum_{j=1}^{n}\delta_i m_j^{[i]}) m_{ik}\bigr)^4\leq\bigl(\E|\sum_{j=1}^{n}\delta_i m_j^{[i]}|^8\bigr)^{\frac{1}{2}}\cdot\bigl(\E|m_{ik}|^{8}\bigr)^{\frac{1}{2}}
       \leq \frac{C}{n^2}+C 1_{\{i=k\}},  
\end{equation*}
where we used \cref{eq:two-point1} in the proof of \cref{two-point} and \cref{eq:lemma3.1adhi} in the last inequality.
From \cref{Delta-error} and \cref{eq:three-onesum1} in the proof of \cref{three-onesum}, we have
\be{
\|\sum_{j=1}^{n}\Delta_i m_{jk}^{[i]}\|_4=O(\frac{1}{\sqrt{n}})
}
and
$$\|m_i\, (\sum_{j=1}^{n}\delta_i m_{jk}^{[i]})\|_4\leq \|\sum_{j=1}^{n}\delta_i m_{jk}^{[i]}\|_4= O(\frac{1}{\sqrt{n}}).$$
Therefore,
\ben{\label{eq:plusindicator}
\|\sum_{j=1}^{n}(m_{jk}-m_{jk}^{(i)})\|_4=O(\frac{1}{\sqrt{n}})+C 1_{\{i=k\}}.
}
From the Cauchy-Schwarz inequality, \cref{eq:plusindicator} and \cref{three-onesum}, we obtain 
\besn{\label{eq:2terms2}
&\E\left(\sum_{j=1}^{n} \left[ \sum_{l=1}^{n} (m_{jl}-m_{jl}^{(i)}) \right] \left[ \sum_{k=1}^{n} m_{ijk} \right]\right)^2
\leq \E \sum_{j_1=1}^{n} \left[ \sum_{l=1}^{n} (m_{j_1 l}-m_{j_1 l}^{(i)}) \right]^2 \sum_{j_2=1}^{n} \left[ \sum_{k=1}^{n} m_{ij_2 k} \right]^2\\
\leq &\sum_{j_1, j_2=1}^n\sqrt{\E \left[ \sum_{l=1}^{n} (m_{j_1 l}-m_{j_1 l}^{(i)}) \right]^4} \cdot \sqrt{\E \left[ \sum_{k=1}^{n} m_{ij_2 k} \right]^4}\\
=&O(n^2\cdot \frac{1}{n} \cdot \frac{1}{n} )+O(n\cdot 1\cdot \frac{1}{n})+O(n\cdot \frac{1}{n}\cdot 1)+O(1\cdot 1\cdot 1)=O(1),
}
where for the case $j_2=i$, 
we also used $m_{iik}=-2m_i m_{ik}$ (cf. \cref{eq:mkkj}) and \cref{two-point}.
Combining \cref{eq:2terms1,eq:2terms2}, we prove \cref{eq:SK-2repeat}. 

It remains to prove \cref{error}.
\begin{proof}[Proof of (\ref{error})]
From the definition of $\Delta_i$ in \cref{eq:defDelta}, we have
\begin{equation*}
\varepsilon_i m_{jk}^{[i]} = m_{jk}^{(i)} + \Delta_i m_{jk}^{[i]}.
\end{equation*}
The desired equation \cref{error} then follows once we prove
    \begin{equation}\label{twopoint-identity}
m_{jk}=\left(\delta_i\,m_j^{[i]}\right)m_{ik}+\varepsilon_i m_{jk}^{[i]}+m_i\left(\delta_i\,m_{jk}^{[i]}\right).
\end{equation}
It remains to prove (\ref{twopoint-identity}). 

Using $\langle\cdot ;\cdot \rangle$ to represent the covariance, we have
\begin{equation}\label{eq:mjkrep1}
m_{jk}
=
\langle \sigma_j;\sigma_k\rangle
:=
\langle \sigma_j\sigma_k\rangle
-
\langle \sigma_j\rangle \langle \sigma_k\rangle.
\end{equation}
By first conditioning on \(\sigma_i\), we have
\begin{equation*}
\langle \sigma_j\sigma_k\rangle
=
\bigl\langle \langle \sigma_j\sigma_k\rangle^{[i]} \bigr\rangle.
\end{equation*}
Under the conditional measure, we have
\begin{equation*}
\langle \sigma_j\sigma_k\rangle^{[i]}
=
\langle \sigma_j;\sigma_k\rangle^{[i]}
+
\langle \sigma_j\rangle^{[i]}\langle \sigma_k\rangle^{[i]}
=
m_{jk}^{[i]} + m_j^{[i]} m_k^{[i]}.
\end{equation*}
Thus,
\begin{equation}\label{eq:mjkrep2}
\langle \sigma_j\sigma_k\rangle
=
\langle m_{jk}^{[i]} \rangle
+
\langle m_j^{[i]} m_k^{[i]} \rangle.
\end{equation}
Again by first conditioning on \(\sigma_i\), we have
\begin{equation}\label{eq:mjkrep3}
\langle \sigma_j\rangle = \langle m_j^{[i]} \rangle,
\qquad
\langle \sigma_k\rangle = \langle m_k^{[i]} \rangle.
\end{equation}
From \cref{eq:mjkrep1,eq:mjkrep2,eq:mjkrep3}, we have
\begin{equation*}
m_{jk}
=
\langle m_{jk}^{[i]} \rangle
+
\left(
\langle m_j^{[i]} m_k^{[i]} \rangle
-
\langle m_j^{[i]} \rangle \langle m_k^{[i]} \rangle
\right),
\end{equation*}
that is,
\begin{equation}\label{eq:mjkrep}
m_{jk}
=
\langle m_j^{[i]} ; m_k^{[i]} \rangle
+
\langle m_{jk}^{[i]} \rangle.
\end{equation}

For any quantity $X^{[i]}(\sigma_i)$ depending only on the value of the spin $\sigma_i$, from \cref{eq:defdelta} and \cref{eq:defeps}, we have the identity
\ben{\label{eq:firep}
X^{[i]}(\sigma_i)=\eps_i X^{[i]}+\sigma_i \delta_i X^{[i]}
}
by verifying it for both the cases $\sigma_i=1$ and  $\sigma_i=-1$.
From \cref{eq:firep}, we have
\begin{equation*}
m_j^{[i]}=\varepsilon_i m_j^{[i]}+\sigma_i\,\delta_i m_j^{[i]},
\qquad
m_k^{[i]}=\varepsilon_i m_k^{[i]}+\sigma_i\,\delta_i m_k^{[i]}.
\end{equation*}
Since \(\varepsilon_i m_j^{[i]}\) and \(\delta_i m_j^{[i]}\) are no longer random, the covariance becomes
\begin{equation*}
\langle m_j^{[i]} ; m_k^{[i]} \rangle
=
(\delta_i m_j^{[i]})(\delta_i m_k^{[i]})
\langle \sigma_i;\sigma_i\rangle=(1-m_i^2)(\delta_i m_j^{[i]})(\delta_i m_k^{[i]}).
\end{equation*}
Using the identity (cf. \cref{eq:mij})
\begin{equation*}
m_{ik}
=
\langle \sigma_i;\sigma_k\rangle
=
(1-m_i^2)\,\delta_i m_k^{[i]},
\end{equation*}
we obtain
\begin{equation}\label{eq:mjkrep4}
\langle m_j^{[i]} ; m_k^{[i]} \rangle
=
(\delta_i m_j^{[i]})\,m_{ik}.
\end{equation}
Taking expectation with respect to the Gibbs measure on both sides of \cref{eq:firep} with $X^{[i]}=m_{jk}^{[i]}$, we have
\begin{equation}\label{eq:mjkrep5}
\langle m_{jk}^{[i]} \rangle
=
\varepsilon_i m_{jk}^{[i]} + m_i\,\delta_i m_{jk}^{[i]}.
\end{equation}
Combining \cref{eq:mjkrep,eq:mjkrep4,eq:mjkrep5}, we obtain \cref{twopoint-identity}.
\end{proof}

\subsection{$\E(\sigma_n^2(H_n))$ and $\Var(\sigma_n^2(H_n))$}\label{subsec:SK3}

We first prove that $\E(\sigma_n^2(H_n))$ is bounded away from 0 for sufficiently large $n$.
By the definition of $\sigma_n^2(H_n)$ and using $m_{ii}=1-m_i^2$, we have
\ben{\label{eq:sigmaH}
\sigma_n^2(H_n)=1-\frac{\sum_{i=1}^n m_i^2}{n}+\frac{\sum_{1\leq i\ne j\leq n} m_{ij}}{n}.
}
From \cref{eq:lineart} in the proof of \cref{two-point}, we have, for sufficiently small $t>0$ and any $i\in \{1,\dots n\}$, 
\ben{\label{eq:sigmaH3}
\E (\sum_{j: j\ne i} m_{ij})^2=O(t).
}
From \cref{eq:sigmaH} and \cref{eq:sigmaH3}, to show that $\E(\sigma_n^2(H_n))$ is bounded away from 0 for sufficiently small $t>0$, it suffices to show that $\E m_i^2$ is bounded away from 1.

From \cite[Lemma~4.1]{adhikari2021dynamical}, we have, for sufficiently small $t>0$,
\ben{\label{eq:lemma41}
\E\left[m_i-\tanh\left(h_i+\sum_{j: j\ne i}g_{ij}m^{(i)}_j\right)\right]^2\leq \frac{C_t}{n}.
}
Note that they only considered the case $h_i\equiv h$ for all $i\in \{1,\dots, n\}$, but it is easy to see that their proof does not rely on this assumption.
From \cref{eq:lemma41}, to show that $\E m_i^2$ is bounded away from 1, it suffices to show that
\be{
\E \tanh^2\left(h_i+\sum_{j: j\ne i}g_{ij}m^{(i)}_j\right)
}
is bounded away from 1. 
This follows by first conditioning on $\{m^{(i)}_j, j\ne i\}$ and then using that $|h_i|\leq c_h$, $|m^{(i)}_j|\leq 1$ and $\{g_{ij}, j\ne i\}$ is independent of $\{m^{(i)}_j, j\ne i\}$.

Next, we prove $\Var(\sigma_n^2(H_n))\to 0$.
From the Gaussian Poincar\'e inequality and a straightforward computation of the partial derivatives, we obtain
\ban{\label{eq:varbound}
& \operatorname{Var}\left(\sigma_n^2\left(H_n\right)\right)=\operatorname{Var}\left(\frac{\sum_{i, j=1}^n m_{i j}}{n}\right) \leq \frac{1}{n^2} \sum_{k, l=1}^n \E\left[\frac{\beta}{\sqrt{n}} \sum_{i, j=1}^n \frac{\partial}{\partial g_{k l}} m_{i j}\right]^2\nonumber \\
= & \frac{\beta^2}{n^3} \sum_{k, l=1}^n \E\left[\sum_{i, j=1}^n\left(m_{i j k l}+m_k m_{i j l}+m_l m_{i j k}+m_{i k} m_{j l}+m_{i l} m_{j k} \right)\right]^2\nonumber \\
\leq & \frac{\beta^2}{n^3} \sum_{k, l=1}^n\left(\sum_{r=1}^5\Big\|A_{k, l}^{(r)}\Big\|_2\right)^2, 
}
where
$m_{ijkl}$ is the fourth mixed cumulant of the spins $\sigma_i, \sigma_j, \sigma_k, \sigma_l$ and
$$
\begin{gathered}
A_{k, l}^{(1)}:=\sum_{i, j=1}^n m_{i j k l}, \quad A_{k, l}^{(2)}:=\sum_{i, j=1}^n m_k m_{i j l}, \quad A_{k, l}^{(3)}:=\sum_{i, j=1}^n m_l m_{i j k} \\
A_{k, l}^{(4)}:=\sum_{i, j=1}^n m_{i k} m_{j l}, \quad A_{k, l}^{(5)}:=\sum_{i, j=1}^n m_{i l} m_{j k}
\end{gathered}
$$

For $A_{k, l}^{(2)}$ and $A_{k, l}^{(3)}$, by \cref{three-point},

$$
\begin{aligned}
\Big\|A_{k, l}^{(2)}\Big\|_2 & =\Big\|\sum_{i, j=1}^n m_k m_{i j l}\Big\|_2 \leq\Big\|\sum_{i, j=1}^n m_{i j l}\Big\|_2 \leq C_t \\
\Big\|A_{k, l}^{(3)}\Big\|_2 & =\Big\|\sum_{i, j=1}^n m_l m_{i j k}\Big\|_2 \leq\Big\|\sum_{i, j=1}^n m_{i j k}\Big\|_2 \leq C_t
\end{aligned}
$$

For $A_{k, l}^{(4)}$ and $A_{k, l}^{(5)}$, by \cref{two-point},

$$
\begin{aligned}
\Big\|A_{k, l}^{(4)}\Big\|_2 & =\Big\|\sum_{i, j=1}^n m_{i k} m_{j l}\Big\|_2 \leq\Big\|\sum_{i=1}^n m_{i k}\Big\|_4 \cdot\Big\|\sum_{j=1}^n m_{j l}\Big\|_4 \leq C_t \\
\Big\|A_{k, l}^{(5)}\Big\|_2 & =\Big\|\sum_{i, j=1}^n m_{i l} m_{j k}\Big\|_2 \leq\Big\|\sum_{i=1}^n m_{i l}\Big\|_4 \cdot\Big\|\sum_{j=1}^n m_{j k}\Big\|_4 \leq C_t
\end{aligned}
$$

Now we consider $\Big\|A_{k, l}^{(1)}\Big\|_2$. 
We will prove that for all $A\subset \{1,\dots, n\}$,
\ben{\label{eq:mijklbound}
\Big\|\sum_{k=1}^n \sum_{l=1}^n m_{i j k l}^{[A]}\Big\|_2 \leq C_t.
}
Differentiating \cref{eq:mijk} with respect to the external field in direction of $\sigma_l$, we obtain
$$
\begin{gathered}
m_{i j k l}^{[A]}=\left(1-\left(m_i^{[A]}\right)^2\right) \delta_i m_{j k l}^{[A \cup\{i\}]}-2 m_i^{[A]} m_{i l}^{[A]} \delta_i m_{j k}^{[A \cup\{i\}]}-2 m_i^{[A]} m_{i k}^{[A]} \delta_i m_{j l}^{[A \cup\{i\}]} \\
-2\left(m_{i l}^{[A]} m_{i k}^{[A]}+m_i^{[A]} m_{i k l}^{[A]}\right) \delta_i m_j^{[A \cup\{i\}]},
\end{gathered}
$$
hence
$$
\sum_{k, l=1}^n m_{i j k l}^{[A]}=\left(1-\left(m_i^{[A]}\right)^2\right) \sum_{k, l=1}^n \delta_i m_{j k l}^{[A \cup\{i\}]}+R_{i j}^{[A]},
$$
where
$$
\begin{aligned}
R_{i j}^{[A]}=- & 2 m_i^{[A]}\left(\sum_{l=1}^n m_{i l}^{[A]}\right)\left(\sum_{k=1}^n \delta_i m_{j k}^{[A \cup\{i\}]}\right)-2 m_i^{[A]}\left(\sum_{k=1}^n m_{i k}^{[A]}\right)\left(\sum_{l=1}^n \delta_i m_{j l}^{[A \cup\{i\}]}\right) \\
& -2 \delta_i m_j^{[A \cup\{i\}]}\left(\sum_{k=1}^n m_{i k}^{[A]}\right)\left(\sum_{l=1}^n m_{i l}^{[A]}\right)-2 m_i^{[A]}\left( \sum_{k, l=1}^n m_{i k l}^{[A]} \right)\delta_i m_j^{[A \cup\{i\}]}  .
\end{aligned}
$$
Using \cref{two-point,three-point,eq:three-onesum1},
$\|R_{i j}^{[A]}\|_2$ can be bounded as
$$
\begin{aligned}
\Big\|R_{i j}^{[A]}\Big\|_2 \leq & C\Big\|\sum_{l=1}^n m_{i l}^{[A]}\Big\|_4\Big\|\sum_{k=1}^n \delta_i m_{j k}^{[A \cup\{i\}]}\Big\|_4 \\
& +C\Big\|\sum_{k=1}^n m_{i k}^{[A]}\Big\|_4\Big\|\sum_{l=1}^n m_{i l}^{[A]}\Big\|_4+C\Big\|\sum_{k, l=1}^n m_{i k l}^{[A]}\Big\|_2 \\
\leq & C_t.
\end{aligned}
$$
Moreover, from \cref{three-point},
$$
\Big\|\left(1-\left(m_i^{[A]}\right)^2\right) \sum_{k, l=1}^n \delta_i m_{j k l}^{[A \cup\{i\}]}\Big\|_2^2 \leq\Big\|\sum_{k, l=1}^n \delta_i m_{j k l}^{[A \cup\{i\}]}\Big\|_2^2 \leq\Big\|\sum_{k, l=1}^n m_{j k l}^{[A \cup\{i\}]}(t)\Big\|_2^2 \leq C_t.
$$
Therefore, we have proved \cref{eq:mijklbound}, hence $\|A_{k,l}^{(1)}\|_2\leq C_t$.

The bound \cref{eq:varbound}, together with the boundedness of $\|A_{k,l}^{(r)}\|_2$, proves that $\Var(\sigma_n^2(H_n))\to 0$.

\subsection{Proofs of lemmas}\label{subsec:SK4}

The proofs are based on the dynamical approach of \cite{adhikari2021dynamical}, in particular, their Lemma~3.1. 
We need various modifications to deal with more complicated functionals of two-point and three-point functions.
In all the following proofs, we fix a given configuration of the spins in $A$.

\paragraph{Proof of Lemma \ref{two-point}.}
Assume $i\notin A$. Otherwise, $m^{[A]}_{ij}=0$ and the lemma holds.
From \cref{eq:mij}, we have
\begin{equation*}
    \sum_{j=1}
    ^n m_{ij}^{[A]}=\left[1-\left(m_i^{[A]}\right)^2\right]\sum_{j=1}
    ^n \delta_i\, m_j^{[A\cup\{i\}]}.
\end{equation*}
Because $|m_i^{[A]}|\leq 1$, to prove the lemma, it suffices to prove
\ben{\label{eq:two-point1}
\E |\sum_{j=1}^n \delta_i m_j^{[A\cup \{i\}]}|^{2+\epsilon}\leq C_{t,\epsilon}
}
for sufficiently small $t$.
We first consider the special cases that $j=i$ or $j\in A$.
Because the spin $\sigma_i$ is fixed in the definition of $m_i^{[A\cup\{i\}]}$, we have
\[
m_i^{[A\cup\{i\}]}=\sigma_i,
\]
hence
\begin{equation*}
\delta_i m_i^{[A\cup\{i\}]}=\frac12\sum_{\sigma_i=\pm1}\sigma_i^2=1.
\end{equation*}
If $j\in A$, then $\sigma_j$ is fixed, $m_j^{[A\cup\{i\}]}=\sigma_j$ and 
\begin{equation*}
\delta_i m_j^{[A\cup\{i\}]}=\frac12\sum_{\sigma_i=\pm1}\sigma_i\,\sigma_j
=\sigma_j\cdot\frac12\sum_{\sigma_i=\pm1}\sigma_i=0.
\end{equation*}
Therefore,
\be{
\sum_{j=1}^n\delta_i m_j^{[A\cup\{i\}]}-1=\sum_{j: j\ne i} \delta_i m_j^{[A\cup\{i\}]}=\sum_{j\notin A\cup \{i\}}\delta_i m_j^{[A\cup\{i\}]}.
}
To prove \cref{eq:two-point1}, it remains to prove
\ben{\label{eq:two-point2}
\E |\sum_{j\notin A\cup \{i\}} \delta_i m_j^{[A\cup \{i\}]}|^{2+\epsilon}\leq C_{t,\epsilon}.
}
From \cref{eq:SI1}, we have
\bes{\label{eq-2-1}
\sum_{j\notin A\cup \{i\}}\delta_i m_j^{[A\cup\{i\}]}=& \sum_{k\notin A}\int_{0}^{t}\sum_{j\notin A\cup \{i\}}\eps_i\, m_{kj}^{[A\cup\{i\}]}(s)\, d g_{ik}(s)\\
&- \sum_{k\notin A}\int_{0}^{t}\sum_{j\notin A\cup \{i\}}\delta_i\!\left( m_k^{[A\cup\{i\}]} m_{kj}^{[A\cup\{i\}]} \right)(s)\,\frac{ds}{n}.
}
From It\^o\!\!'s lemma (cf. \cite[Page~10]{adhikari2021dynamical}) and noting that $m_{ij}^{[A\cup \{i\}]}=0$, we have
\begin{equation}\label{eq:J1J2}
    \begin{split}
        &\E\bigl|\sum_{j\notin A\cup \{i\}}\delta_i m_j^{[A\cup\{i\}]}\bigr|^{2+\epsilon}(t)\\
\le& \left(1+\frac{\epsilon}{2}\right)(1+\epsilon)
\sum_{k\notin A\cup\{i\}}\int_{0}^{t}
\E\bigl|\sum_{j\notin A\cup \{i\}}\delta_i m_j^{[A\cup\{i\}]}\bigr|^{\epsilon}(s)\,
\bigl|\eps_i\sum_{j\notin A\cup \{i\}}m_{kj}^{[A\cup\{i\}]}\bigr|^{2}(s)\,\frac{ds}{n}\\
&\qquad
+(2+\epsilon)\sum_{k\notin A\cup\{i\}}\int_{0}^{t}
\E\bigl|\sum_{j\notin A\cup \{i\}}\delta_i m_j^{[A\cup\{i\}]}\bigr|^{1+\epsilon}(s)\,
\bigl|\delta_i\!\left(m_k^{[A\cup\{i\}]}\sum_{j\notin A\cup \{i\}}m_{kj}^{[A\cup\{i\}]}\right)\bigr|(s)\,\frac{ds}{n}\\
=:&J_1+J_2.
    \end{split}
\end{equation}
By Young's inequality (i.e., $a^{\theta} b^{1-\theta} \le \theta a + (1-\theta)b$), we have
\begin{equation*}
    \begin{split}
        J_1&\leq \frac{\epsilon(1+\epsilon)}{2}\sum_{k\notin A\cup\{i\}}\int_{0}^{t}\E\bigl|\sum_{j\notin A\cup\{i\}}\delta_i m_j^{[A\cup\{i\}]}\bigr|^{2+\epsilon}(s)\,\frac{ds}{n}\\
        &\quad+(1+\epsilon)\sum_{k\notin A\cup\{i\}}\int_{0}^{t} \E\bigl|\eps_i\sum_{j\notin A\cup\{i\}}m_{kj}^{[A\cup\{i\}]}\bigr|^{2+\epsilon}(s)\,\frac{ds}{n}\\
        &\leq \frac{\epsilon(1+\epsilon)}{2}\sum_{k\notin A\cup\{i\}}\int_{0}^{t}\E\bigl|\sum_{j\notin A\cup\{i\}}\delta_i m_j^{[A\cup\{i\}]}\bigr|^{2+\epsilon}(s)\,\frac{ds}{n}\\
        &\quad +(1+\epsilon)2^{1+\epsilon}\sum_{k\notin A\cup\{i\}}\int_{0}^{t} \E\bigl|\eps_i\sum_{j\notin A\cup\{i,k\}}m_{kj}^{[A\cup\{i\}]}\bigr|^{2+\epsilon}(s)\,\frac{ds}{n}\\
        &\quad +(1+\epsilon)2^{1+\epsilon}\sum_{k\notin A\cup\{i\}}\int_{0}^{t} \E\bigl|\eps_i m_{kk}^{[A\cup\{i\}]}\bigr|^{2+\epsilon}(s)\,\frac{ds}{n}.
    \end{split}
\end{equation*}
Using
$0\le \varepsilon_i\, m_{kk}^{[A\cup\{i\}]}\le 1$,
we have,
\begin{equation*}
\sum_{k\notin A\cup\{i\}}\int_0^t \frac{1}{n}\,
\E\Big|\varepsilon_i\, m_{kk}^{[A\cup\{i\}]}\Big|^{2+\epsilon}(s)\,ds
\le \frac{n-|A|-1}{n}\,t\le t.
\end{equation*}
Besides,
\begin{equation*}
    \begin{split}
        \sum_{k\notin A\cup\{i\}}\int_{0}^{t} \E\bigl|\varepsilon_i\sum_{j\notin A\cup\{i, k\}}m_{kj}^{[A\cup\{i\}]}\bigr|^{2+\epsilon}(s)\,\frac{ds}{n}&\leq \sup_{k\notin A\cup\{i\}}\int_{0}^{t} \E\bigl|\varepsilon_i\sum_{j\notin A\cup\{i,k\}}m_{kj}^{[A\cup\{i\}]}\bigr|^{2+\epsilon}(s)\,ds\\
        &\leq  \sup_{k\notin A\cup\{i\}}\int_{0}^{t} \E\varepsilon_i\bigl|\sum_{j\notin A\cup\{i,k\}}m_{kj}^{[A\cup\{i\}]}\bigr|^{2+\epsilon}(s)\,ds\\
        &\leq \sup_{k\notin A\cup\{i\}, \sigma_i=\pm 1}\int_{0}^{t} \E\bigl|\sum_{j\notin A\cup\{i,k\}}m_{kj}^{[A\cup\{i\}]}\bigr|^{2+\epsilon}(s)\,ds.
    \end{split}
\end{equation*}
Combining the above three inequalities, we obtain
\bes{
    J_1\leq& (1+\epsilon)2^{1+\epsilon}t+ \frac{\epsilon(1+\epsilon)}{2}\int_{0}^{t}\E\bigl|\sum_{j\notin A\cup\{i\}}\delta_i m_j^{[A\cup\{i\}]}\bigr|^{2+\epsilon}(s)\, ds \\
    &+(1+\epsilon)2^{1+\epsilon}\sup_{k\notin A\cup\{i\}, \sigma_i=\pm 1}\int_{0}^{t} \E\bigl|\sum_{j\notin A\cup\{i,k\}}m_{kj}^{[A\cup\{i\}]}\bigr|^{2+\epsilon}(s)\,ds.
}
Similarly, 
\begin{equation*}
    \begin{split}
        J_2&\leq(1+\epsilon)\sum_{k\notin A\cup\{i\}}\int_{0}^{t}
\E\bigl|\sum_{j\notin A\cup\{i\}}\delta_i m_j^{[A\cup\{i\}]}\bigr|^{2+\epsilon}(s)\,\frac{ds}{n}\\
&\quad +\sum_{k\notin A\cup\{i\}}\int_{0}^{t}
\bigl|\delta_i\!\left(m_k^{[A\cup\{i\}]}\sum_{j\notin A\cup\{i\}}m_{kj}^{[A\cup\{i\}]}\right)\bigr|^{2+\epsilon}(s)\,\frac{ds}{n}\\
&\leq(1+\epsilon)\int_{0}^{t}
\E\bigl|\sum_{j\notin A\cup\{i\}}\delta_i m_j^{[A\cup\{i\}]}\bigr|^{2+\epsilon}(s)\,ds+2\sum_{k\notin A\cup\{i\}}\int_{0}^{t}
\bigl|\delta_i\!\left(m_k^{[A\cup\{i\}]}m_{kk}^{[A\cup\{i\}]}\right)\bigr|^{2+\epsilon}(s)\,\frac{ds}{n}\\
&\quad +2\sum_{k\notin A\cup\{i\}}\int_{0}^{t}
\bigl|\delta_i\!\left(m_k^{[A\cup\{i\}]}\sum_{j\notin A\cup\{i,k\}}m_{kj}^{[A\cup\{i\}]}\right)\bigr|^{2+\epsilon}(s)\,\frac{ds}{n},\\
&\leq   2t+ (1+\epsilon)\int_{0}^{t}\E\bigl|\sum_{j\notin A\cup\{i\}}\delta_i m_j^{[A\cup\{i\}]}\bigr|^{2+\epsilon}(s)\, ds \\
&\quad+2\sup_{k\notin A\cup\{i\}, \sigma_i=\pm 1}\int_{0}^{t} \E\bigl|m_k^{[A\cup\{i\}]}\bigr|^{2+\epsilon}\bigl|\sum_{j\notin A\cup\{i, k\}}m_{kj}^{[A\cup\{i\}]}\bigr|^{2+\epsilon}(s)\,ds.
    \end{split}
\end{equation*}
Combining the upper bounds for $J_1$ and $J_2$ above, we obtain
\begin{equation}\label{eq:557}
\begin{split}
    J_1+J_2&\le t C_{\epsilon}+(1+\epsilon)(1+\frac{\epsilon}{2})\int_{0}^{t}\E\bigl|\sum_{j\notin A\cup\{i\}}\delta_i m_j^{[A\cup\{i\}]}\bigr|^{2+\epsilon}(s)\, ds\\
    &\quad +(1+\epsilon)2^{1+\epsilon}\sup_{k\notin A\cup\{i\}, \sigma_i=\pm 1}\int_{0}^{t} \E\bigl(1+\bigl|m_k^{[A\cup\{i\}]}\bigr|^{2+\epsilon}\bigr)\bigl|\sum_{j\notin A\cup\{i,k\}}m_{kj}^{[A\cup\{i\}]}\bigr|^{2+\epsilon}(s)\,ds.
    \end{split}
\end{equation}
Inserting \cref{eq:mij}, we have
\begin{equation*}
    \begin{split}
        \bigl(1+\lvert m_k^{[A\cup\{i\}]}\rvert^{2+\epsilon}\bigr)\,
\bigl\lvert \sum_{j\notin A\cup\{i,k\} }m_{kj}^{[A\cup\{i\}]}(\sigma_i)\bigr\rvert^{2+\epsilon}
&\leq \bigl[1-\lvert m_k^{[A\cup\{i\}]}\rvert^{2}\bigr]^{1+\epsilon}
\bigl\lvert \sum_{j\notin A \cup\{i,k\}}\delta_k m_j^{[A\cup\{i,k\}]}(\sigma_i)\bigr\rvert^{2+\epsilon}\\
&\leq \bigl\lvert \sum_{j\notin A\cup\{i,k\} }\delta_k m_j^{[A\cup\{i,k\}]}(\sigma_i)\bigr\rvert^{2+\epsilon}.
    \end{split}
\end{equation*}
Therefore,
\begin{equation*}
    \begin{split}
        \E\bigl|\sum_{j\notin A\cup\{i\}}\delta_i m_j^{[A\cup\{i\}]}\bigr|^{2+\epsilon}(t)&\leq  t C_\epsilon+(1+\epsilon)(1+\frac{\epsilon}{2})\int_{0}^{t}\E\bigl|\sum_{j\notin A\cup\{i\}}\delta_i m_j^{[A\cup\{i\}]}\bigr|^{2+\epsilon}(s)\, ds\\
        &\quad +(1+\epsilon)2^{1+\epsilon}\sup_{k\notin A\cup\{i\}, \sigma_i=\pm 1}\int_{0}^{t} \E\bigl|\delta_k\sum_{j\notin A\cup\{i,k\}}m_{j}^{[A\cup\{i,k\}]}\bigr|^{2+\epsilon}(s)\,ds.
    \end{split}
\end{equation*}
By Gronwall's inequality (cf. \cite[Page~10]{adhikari2021dynamical}), we have
\bes{
&\E\bigl|\sum_{j\notin A\cup\{i\}}\delta_i\,m_j^{[A\cup\{i\}]}\bigr|^{2+\epsilon}(t)\\
\le&
t C_{\epsilon}
+
(1+\epsilon)2^{1+\epsilon}\sup_{k\notin A\cup\{i\},\ \sigma_i=\pm1}
\int_0^t
e^{(1+\epsilon)(1+\epsilon/2)(t-s)}\E\bigl|
\delta_k\sum_{j\notin A\cup\{i,k\}}m_j^{[A\cup\{i,k\}]}
\bigr|^{2+\epsilon}(s)\,ds
.
}
By iterating this estimate $m=n-|A|$ times (cf. \cite[Page~11]{adhikari2021dynamical}), we obtain
\begin{equation}\label{eq:lineart}
    \begin{split}
   &  \E\bigl|\sum_{j\notin A\cup\{i\}}\delta_i\,m_j^{[A\cup\{i\}]}\bigr|^{2+\epsilon}(t)\\
&\le   t C_{\epsilon}\bigl[1+\frac{(1+\epsilon)2^{1+\epsilon}(e^{(1+\epsilon)(1+\epsilon/2)t}-1)}{(1+\epsilon)(1+\epsilon/2)}+\cdots+\frac{(1+\epsilon)^m2^{(1+\epsilon)m}(e^{(1+\epsilon)(1+\epsilon/2)t}-1)^m}{(1+\epsilon)^m(1+\epsilon/2)^m}\bigr].
    \end{split}
\end{equation}
For sufficiently small $t>0$, the above geometric series converge. Therefore, we obtain \cref{eq:two-point2}.
\qed

\

The following proofs are similar to the proof of \cref{two-point}, and we will omit some analogous arguments for brevity.

\paragraph{Proof of Lemma \ref{three-onesum}.}
From \cref{eq:mijk}, we have
\begin{equation}\label{eq-three-onesum}
\sum_{j=1}^n m_{ijk}^{[A]}
=\left[1-\left(m_i^{[A]}\right)^2\right]\bigl(\sum_{j=1}^n\delta_i m_{jk}^{[A\cup\{i\}]}\bigr)
-2\,m_i^{[A]}m_{ik}^{[A]}\,\bigl(\sum_{j=1}^n\delta_i m_j^{[A\cup\{i\}]}\bigr).
\end{equation}
The second term can be bounded using \cref{eq:lemma3.1adhi} and \cref{eq:two-point2} as
\begin{equation*}
    \begin{split}
\bigl\|m_i^{[A]}m_{ik}^{[A]}\,\bigl(\sum_{j=1}^n\delta_i m_j^{[A\cup\{i\}]}\bigr)\bigr\|_{2+\epsilon}\leq \bigl\|m_{ik}^{[A]}\,\bigr\|_{4+2\epsilon}\bigl\|\bigl(\sum_{j=1}^n\delta_i m_j^{[A\cup\{i\}]}\bigr)\bigr\|_{4+2\epsilon}\le \frac{C_{t,\epsilon}}{\sqrt{n}}.
    \end{split}
\end{equation*}
Therefore, it remains to prove
\begin{equation}\label{eq:three-onesum1}
    \E|\sum_{j=1}^n\delta_i m_{jk}^{[A\cup\{i\}]}|^{2+\epsilon}\le \frac{C_{t,\epsilon}}{n^{1+\epsilon/2}}.
\end{equation}
From \cref{eq:SI2}, we have
\begin{equation}\label{eq:sumSI2}
\begin{split}
    \sum_{j=1}^n\delta_i m_{jk}^{[A\cup\{i\}]}&= \sum_{j\notin A\cup\{i\}}\delta_i m_{jk}^{[A\cup\{i\}]}\\
&= \sum_{l\notin A}\int_{0}^{t}\varepsilon_i\, \sum_{j\notin A\cup\{i\}} m_{jkl}^{[A\cup\{i\}]}(s)\, d g_{il}(s) -\sum_{l\notin A}\int_{0}^{t}\delta_i\! \sum_{j\notin A\cup\{i\}}\left(m_{kl}^{[A\cup\{i\}]}m_{jl}^{[A\cup\{i\}]}\right)(s)\,\frac{ds}{n}\\
&\quad -\sum_{l\notin A}\int_{0}^{t}\delta_i\! \sum_{j\notin A\cup\{i\}}\left(m_{l}^{[A\cup\{i\}]}m_{jkl}^{[A\cup\{i\}]}\right)(s)\,\frac{ds}{n}.
\end{split}
\end{equation}
From It\^o\!\!'s lemma, we have
\begin{align*}
&\E|\sum_{j\notin A\cup\{i\}}\delta_i m_{jk}^{[A\cup\{i\}]}|^{2+\epsilon}(t)\\
&\leq \left(1+\frac{\epsilon}{2}\right)(1+\epsilon)\sum_{l\notin A}\int_0^t \E\big|\sum_{j\notin A\cup\{i\}}\delta_i m_{jk}^{[A\cup\{i\}]}\big|^{\epsilon}(s) \big| \varepsilon_i \sum_{j\notin A\cup\{i\}}m^{[A\cup\{i\}]}_{jkl}\big|^2(s)\frac{ds}{n}\\
&\quad
+(2+\epsilon)\sum_{l\notin A}\int_0^t \E\big|\sum_{j\notin A\cup\{i\}}\delta_i m_{jk}^{[A\cup\{i\}]}\big|^{1+\epsilon}(s)\big|\delta_i\!\big(m^{[A\cup\{i\}]}_{kl}  \sum_{j\notin A\cup\{i\}}m^{[A\cup\{i\}]}_{jl}\big)\big|(s)\frac{ds}{n}\\
&\quad
+(2+\epsilon)\sum_{l\notin A}\int_0^t \E\big|\sum_{j\notin A\cup\{i\}}\delta_i m_{jk}^{[A\cup\{i\}]}\big|^{1+\epsilon}(s) \big|\delta_i\!\big(m^{[A\cup\{i\}]}_{l} \sum_{j\notin A\cup\{i\}}m^{[A\cup\{i\}]}_{jkl}\big)\big|(s)\frac{ds}{n}\\
&=:J_1+J_2+J_3.
\end{align*}
As in bounding \cref{eq:J1J2}, but separating the cases $l\ne k$ and $l=k$, we have
\begin{equation*}
    \begin{split}
        J_1&\leq 
        \epsilon(1+\epsilon) \int_0^t \E \big|\sum_{j\notin A\cup \{i\}}\delta_i m_{jk}^{[A\cup \{i\}]} \big|^{2+\ep}(s)ds+(1+\ep)\sum_{l\notin A\cup \{k\}} \int_0^t \E\big|\eps_i\sum_{j\notin A\cup\{i\}} m_{jkl}^{[A\cup \{i\}]} \big|^{2+\ep}(s) \frac{ds}{n}\\
        &\quad +(1+\ep)\int_0^t \E\big|\eps_i\sum_{j\notin A\cup\{i\}} m_{jkk}^{[A\cup \{i\}]} \big|^{2+\ep}(s) \frac{ds}{n^{1+\ep/2}} \\
        &\leq \frac{C_{t,\epsilon}}{n^{1+\ep/2}}+\epsilon(1+\epsilon) \int_0^t \E \big|\sum_{j\notin A\cup \{i\}}\delta_i m_{jk}^{[A\cup \{i\}]} \big|^{2+\ep}(s)ds \\
        &\quad +(1+\ep)\sup_{l\notin A\cup \{k\},\sigma_i=\pm 1}\int_0^t \E\big| \sum_{j\notin A\cup\{i\}}m^{[A\cup\{i\}]}_{jkl}\big|^{2+\eps}(s)ds,
    \end{split}
\end{equation*}
where we used (cf. \cref{eq:mkkj})
\be{
m_{jkk}^{[A\cup \{i\}]}=-2m_k^{[A\cup \{i\}]} m_{kj}^{[A\cup \{i\}]}
}
and \cref{two-point} in the last inequality.
For $J_2$ and $J_3$, similarly, we have
\begin{equation*}
    \begin{split}
        J_2
&\leq \frac{C_{t,\epsilon}}{n^{1+\ep/2}}+2(1+\ep) \int_0^t\E|\sum_{j\notin A\cup\{i\}}\delta_i m_{jk}^{[A\cup\{i\}]}|^{2+\ep}(s)\,ds\\
&\quad+\sup_{l\notin A\cup \{k\},\sigma_i=\pm 1}\int_0^t \E|m^{[A\cup\{i\}]}_{kl}  \sum_{j\notin A\cup\{i\}}m^{[A\cup\{i\}]}_{jl}|^{2+\ep}(s)ds\\
& \leq \frac{C_{t,\epsilon}}{n^{1+\eps/2}}+ \int_0^t\E|\sum_{j\notin A\cup\{i\}}\delta_i m_{jk}^{[A\cup\{i\}]}|^{2+\ep}(s)\,ds
    \end{split}
\end{equation*}
and
\begin{equation*}
    \begin{split}
        J_3
&\leq \frac{C_{t,\epsilon}}{n^{1+\eps/2}}+2(1+\eps)\int_0^t\E|\sum_{j\notin A\cup\{i\}}\delta_i m_{jk}^{[A\cup\{i\}]}|^{2+\eps}(s)\,ds\\
&\quad +\sup_{l\notin A\cup \{k\},\sigma_i=\pm 1}\int_0^t \E|m^{[A\cup\{i\}]}_{l}  \sum_{j\notin A\cup\{i\}}m^{[A\cup\{i\}]}_{jkl}|^{2+\eps}(s) ds.
    \end{split}
\end{equation*}
Combining the bounds on $J_1$--$J_3$, we obtain
\begin{equation*}
\begin{split}
    \E|\sum_{j\notin A\cup\{i\}}\delta_i m_{jk}^{[A\cup\{i\}]}|^{2+\ep}(t)&\le \frac{C_{t,\epsilon}}{n^{1+\eps/2}}+(4+\ep)(1+\ep)\int_0^t\E|\sum_{j\notin A\cup\{i\}}\delta_i m_{jk}^{[A\cup\{i\}]}|^{2+\ep}(s)\,ds\\
    &\quad +\sup_{l\notin A\cup \{k\},\sigma_i=\pm 1}\int_0^t \E\!\big(1+|m^{[A\cup\{i\}]}_{l}|\big)\big| \sum_{j\notin A\cup\{i\}}m^{[A\cup\{i\}]}_{jkl}\big|^{2+\ep}(s)ds\\
    &\le \frac{C_{t,\epsilon}}{n^{1+\eps/2}}+(4+\ep)(1+\ep)\int_0^t\E|\sum_{j\notin A\cup\{i\}}\delta_i m_{jk}^{[A\cup\{i\}]}|^{2+\ep}(s)\,ds\\
    &\quad +\sup_{l\notin A\cup \{k\},\sigma_i=\pm 1}\int_0^t \E\big| \sum_{j\notin A\cup\{i\}}\delta_l m^{[A\cup\{i\}]}_{jk}\big|^{2+\ep}(s)ds.
    \end{split}
\end{equation*}
where we used \cref{eq:mijk}, \cref{eq:lemma3.1adhi} and \cref{eq:two-point1} in the last inequality.
The lemma is proved using the same argument at the end of the proof of \cref{two-point} via Gronwall's inequality and iteration.
\qed

\paragraph{Proof of lemma \ref{three-point}.}
From \cref{eq:mijk}, we have
\begin{equation*}
\sum_{k=1}^n\sum_{j=1}^nm_{ijk}^{[A]}
=\left[1-\left(m_i^{[A]}\right)^2\right]\sum_{k,j=1}^n \delta_i m_{jk}^{[A\cup\{i\}]}
-2\,m_i^{[A]}\bigl(\sum_{k=1}^nm_{ik}^{[A]}\,\bigr)\bigl(\sum_{j=1}^n\delta_i m_j^{[A\cup\{i\}]}\bigr).
\end{equation*}
The second term can be bounded using \cref{two-point} and \cref{eq:two-point1} as
\begin{equation}\label{four moment}
        \bigl\|\bigl(\sum_{k=1}^nm_{ik}^{[A]}\,\bigr)\bigl(\sum_{j=1}^n\delta_i m_j^{[A\cup\{i\}]}\bigr)\bigr\|_{2+\eps}
\le \bigl\|\sum_{k=1}^nm_{ik}^{[A]}\bigr\|_{4+\ep}\,\bigl\|\sum_{j=1}^n\delta_i m_j^{[A\cup\{i\}]}\bigr\|_{4+2\ep}= O(1).
\end{equation}
It remains to control $\sum_{k,j=1}^n\delta_i m_{jk}^{[A\cup\{i\}]}$.
From \cref{eq:SI2}, we have
\begin{equation}\label{eq:sumjkdeltai}
\begin{split}
    \sum_{k,j=1}^n\delta_i m_{jk}^{[A\cup\{i\}]}
&= \sum_{l\notin A}\int_{0}^{t}\varepsilon_i\, \sum_{k,j=1}^n m_{jkl}^{[A\cup\{i\}]}(s)\, d g_{il}(s)\\
&\quad -\sum_{l\notin A}\int_{0}^{t}\delta_i\! \sum_{k,j=1}^n\left(m_{kl}^{[A\cup\{i\}]}m_{jl}^{[A\cup\{i\}]}\right)(s)\,\frac{ds}{n}\\
&\quad -\sum_{l\notin A}\int_{0}^{t}\delta_i\! \sum_{k,j=1}^n\left(m_{l}^{[A\cup\{i\}]}m_{jkl}^{[A\cup\{i\}]}\right)(s)\,\frac{ds}{n}.
\end{split}
\end{equation}
Similarly to bounding \cref{eq:sumSI2}, from It\^o\!\!'s lemma, we have
\begin{align*}
&\E|\sum_{k,j=1}^n\delta_i m_{jk}^{[A\cup\{i\}]}|^{2+\eps}\\
&=(1+\frac{\ep}{2})(1+\ep)\sum_{l\notin A}\int_0^t \E\big|\sum_{k,j=1}^n\delta_i m_{jk}^{[A\cup\{i\}]}\big|^\ep(s)\big| \varepsilon_i  \sum_{k,j=1}^n m^{[A\cup\{i\}]}_{jkl}\big|^2(s)\frac{ds}{n}\\
&\quad
+(2+\ep)\sum_{l\notin A}\int_0^t \E \big|\sum_{k,j=1}^n\delta_i m_{jk}^{[A\cup\{i\}]} \big|^{1+\ep}(s)\big|\delta_i\!\big(\sum_{k=1}^n m^{[A\cup\{i\}]}_{kl}  \sum_{j=1}^n m^{[A\cup\{i\}]}_{jl}\big)\big|(s)\frac{ds}{n}\\
&\quad
+(2+\ep)\sum_{l\notin A}\int_0^t \E \big|\sum_{k,j=1}^n\delta_i m_{jk}^{[A\cup\{i\}]} \big|^{1+\ep}(s)\big|\delta_i\!\big(m^{[A\cup\{i\}]}_{l}   \sum_{k,j=1}^n m^{[A\cup\{i\}]}_{jkl}\big)\big|(s)\frac{ds}{n}\\
&=:J_1+J_2+J_3,
\end{align*}
where, using \cref{two-point} in bounding $J_2$,
\begin{equation*}
    \begin{split}
        J_1&\leq 
        \frac{\epsilon(1+\epsilon)}{2} \int_0^t \E \big|\sum_{k,j=1}^n\delta_i m_{jk}^{[A\cup \{i\}]} \big|^{2+\ep}(s)ds+(1+\ep)\sum_{l\notin A} \int_0^t \E\big|\eps_i\sum_{k,j=1}^n m_{jkl}^{[A\cup \{i\}]} \big|^{2+\ep}(s) \frac{ds}{n}\\
        &\leq 
        \frac{\epsilon(1+\epsilon)}{2} \int_0^t \E \big|\sum_{k,j=1}^n\delta_i m_{jk}^{[A\cup \{i\}]} \big|^{2+\ep}(s)ds+(1+\ep)\sup_{l\notin A ,\sigma_i=\pm 1} \int_0^t \E\big|\sum_{k,j=1}^n m_{jkl}^{[A\cup \{i\}]} \big|^{2+\ep}(s) ds,
    \end{split}
\end{equation*}
\begin{equation*}
    \begin{split}
        J_2
&\leq (1+\ep) \int_0^t\E|\sum_{k,j=1}^n\delta_i m_{jk}^{[A\cup\{i\}]}|^{2+\ep}(s)\,ds+\sup_{l\notin A,\sigma_i=\pm 1}\int_0^t \E |\sum_{k=1}^n m^{[A\cup\{i\}]}_{kl}  \sum_{j=1}^n m^{[A\cup\{i\}]}_{jl}|^{2+\ep}(s)ds\\
& \leq C_{t,\epsilon}+ (1+\ep) \int_0^t\E|\sum_{k,j=1}^n\delta_i m_{jk}^{[A\cup\{i\}]}|^{2+\ep}(s)\,ds
    \end{split}
\end{equation*}
and
\begin{equation*}
    \begin{split}
        J_3
&\leq (1+\eps)\int_0^t\E|\sum_{k,j=1}^n\delta_i m_{jk}^{[A\cup\{i\}]}|^{2+\eps}(s)\,ds +\sup_{l\notin A,\sigma_i=\pm 1}\int_0^t \E|\sum_{k,j=1}^n m^{[A\cup\{i\}]}_{jkl}|^{2+\eps}(s) ds.
    \end{split}
\end{equation*}
Hence,
\begin{equation*}
\begin{split}
    \E|\sum_{k,j=1}^n \delta_i m_{jk}^{[A\cup\{i\}]}|^{2+\ep}&\le C_{t,\ep}+(1+\ep)(2+\frac{\ep}{2})\int_0^t\E|\sum_{k,j=1}^n \delta_i m_{jk}^{[A\cup\{i\}]}|^{2+\ep}(s)\,ds\\
    &\quad +(2+\ep)\sup_{l\notin A,\sigma_i=\pm 1}\int_0^t \E \big| \sum_{k,j=1}^n m^{[A\cup\{i\}]}_{jkl}\big|^2(s)ds.
    \end{split}
\end{equation*}
The remaining proof is the same as the end of the proof of lemma \ref{three-onesum}.
\qed

\paragraph{Proof of lemma \ref{Delta-error}.}
From a similar representation as \cref{eq:SI2}, we have (cf. \cref{eq:sumSI2})
\begin{equation*}
\begin{split}
    &\sum_{j=1}^n\Delta_i m_{jk}^{[A\cup\{i\}]}= \sum_{j\notin A\cup\{i\}}\Delta_i m_{jk}^{[A\cup\{i\}]}\\
&= \sum_{l\notin A}\int_{0}^{t}\delta_i\, \sum_{j\notin A\cup\{i\}} m_{jkl}^{[A\cup\{i\}]}(s)\, d g_{il}(s) -\sum_{l\notin A}\int_{0}^{t}\eps_i\! \sum_{j\notin A\cup\{i\}}\left(m_{kl}^{[A\cup\{i\}]}m_{jl}^{[A\cup\{i\}]}\right)(s)\,\frac{ds}{n}\\
&\quad -\sum_{l\notin A}\int_{0}^{t}\eps_i\! \sum_{j\notin A\cup\{i\}}\left(m_{l}^{[A\cup\{i\}]}m_{jkl}^{[A\cup\{i\}]}\right)(s)\,\frac{ds}{n}.
\end{split}
\end{equation*}
The remaining proof is similar to the proof of \cref{three-onesum} and hence omitted.
\qed

\paragraph{Proof of \cref{Delta-error-higher-order}.}
Assume $i\notin A$. Otherwise, the lemma is trivial.
By the identity (cf. \cref{error})
\[
m_{jk}^{(A)} - m_{jk}^{(A \cup \{i\})}
=
\left( \delta_i m_j^{[i,A]} \right) m_{ik}^{(A)}
+ \Delta_i m_{jk}^{[i,A]}
+ m_i^{(A)} \left( \delta_i m_{jk}^{[i,A]} \right),
\]
we have
\begin{align*}
    \sum_{l=1}^n(m_{lj}^{(A)}-m_{lj}^{(A\cup\{i\})})=(\delta_i\sum_{l=1}^nm_l^{[i,A]})m_{ij}^{(A)}+\Delta_i\sum_{l=1}^nm_{lj}^{[i,A]}+m_i^{(A)}(\delta_i\sum_{l=1}^nm_{lj}^{[i,A]}).
\end{align*}
Then, from \cref{eq:two-point2}, \cref{eq:lemma3.1adhi}, \cref{Delta-error} and \cref{eq:three-onesum1} (it can be seen from their proofs that they hold even 
if the superscript $[A]$ is replaced by $[A,B]$, that is, all the spins in $B$ are removed), we obtain
\begin{align*}
    \|(a_{j}^{(A)}-a_{j}^{(A\cup \{i\})})(t)\|_{2+\ep}\le&
    \|(\delta_i\sum_{l=1}^nm_l^{[i,A]})m_{ij}^{(A)}\|_{2+\ep}+
    \|\Delta_i\sum_{l=1}^nm_{lj}^{[i,A]}\|_{2+\ep}+
    \|(\delta_i\sum_{l=1}^nm_{lj}^{[i,A]})\|_{2+\ep}\\
    \le& \frac{C_{t,\ep}}{\sqrt{n}}.
\end{align*}
\qed

\paragraph{Proof of lemma \ref{two-point-coef}.}
Recall the definition of $a_j$ from \cref{Delta-error-higher-order}.
    Let (recall the fact that $a^{(A\cup \{i\})}_j=0$ if $j\in A\cup \{i\}$)
    \be{
    X_t:=\sum_{j=1}^na_j^{(A\cup\{i\})}\delta_im_j^{[A\cup\{i\}]}(t)=\sum_{j\notin A\cup \{i\}} a_j^{(A\cup\{i\})}\delta_im_j^{[A\cup\{i\}]}(t).
    } 
    From \cref{eq:SI1}, we have
    \begin{align*}
        X_t= &\sum_{k\notin A}\int_{0}^{t}\sum_{j\notin A\cup \{i\}} a_j^{(A\cup\{i\})}\varepsilon_i\, m_{kj}^{[A\cup\{i\}]}(s)\, d g_{ik}(s)\\
&- \sum_{k\notin A}\int_{0}^{t}\sum_{j\notin A\cup \{i\}} a_j^{(A\cup\{i\})}\delta_i\!\left( m_k^{[A\cup\{i\}]} m_{kj}^{[A\cup\{i\}]} \right)(s)\,\frac{ds}{n}.
    \end{align*}
Similar to \cref{eq:J1J2}, from It\^o\!\!'s lemma, we have
\begin{align*}
    \E|X_t|^{2+\ep}\le&(1+\frac{\ep}{2})(1+\ep)\sum_{k\notin A\cup \{i\}}\int_0^t\E|X_s|^{\ep}|\sum_{j\notin A\cup \{i\}} a_j^{(A\cup\{i\})}\varepsilon_i\, m_{kj}^{[A\cup\{i\}]}|^2(s)\frac{ds}{n}\\
    &+(2+\ep)\sum_{k\notin A\cup\{i\}}\int_0^t\E|X_s|^{1+\ep}|\sum_{j\notin A\cup \{i\}} a_j^{(A\cup\{i\})}\delta_i\!\left( m_k^{[A\cup\{i\}]} m_{kj}^{[A\cup\{i\}]} \right)|(s)\frac{ds}{n}\\
    =:&J_{1}+J_{2}.
\end{align*}
From Young's inequality and \cref{eq:mij}, we have
\begin{align*}
    J_1
    \le&\frac{\ep(1+\ep)}{2}\int_{0}^t\E|X_s|^{2+\ep}ds
    +(1+\ep)\sum_{k\notin A\cup \{i\}}\int_{0}^t\E|\sum_{j\notin A\cup \{i\}} a_j^{(A\cup\{i\})}\varepsilon_i\, m_{kj}^{[A\cup\{i\}]}|^{2+\ep}(s)\frac{ds}{n}\\
    \le&\frac{\ep(1+\ep)}{2}\int_{0}^t\E|X_s|^{2+\ep}ds
    +(1+\ep)\sum_{k\notin A\cup \{i\}}\int_{0}^t\E|\sum_{j\notin A\cup \{i\}} a_j^{(A\cup\{i\})}\delta_k\, m_{j}^{[A\cup\{i,k\}]}|^{2+\ep}(s)\frac{ds}{n}\\
    \le&\frac{\ep(1+\ep)}{2}\int_{0}^t\E|X_s|^{2+\ep}ds
    +
    (1+\ep)2^{1+\ep}\sup_{\substack{k\notin A\cup \{i\}\\\sigma_i=\pm1}}\int_0^t\|\sum_{j\notin A\cup \{i\}} a_j^{(A\cup\{i,k\})}\delta_k m_j^{[A\cup \{i,k\}]}(s)\|_{2+\ep}^{2+\ep}ds\\
    &+
    (1+\ep)2^{1+\ep}\sup_{\substack{k\notin A\cup \{i\}\\\sigma_i=\pm1}}\int_0^t\|\sum_{j\notin A\cup \{i\}}(a_j^{(A\cup\{i\})}-a_j^{(A\cup\{i,k\})})\delta_k m_j^{[A\cup \{i,k\}]}(s)\|_{2+\ep}^{2+\ep}ds.
\end{align*}
For the third term above, we have, from \cref{Delta-error-higher-order} and \cref{eq:lemma3.1adhi},
\begin{align*}
    &\int_0^t\|\sum_{j\notin A\cup \{i,k\}}(a_j^{(A\cup\{i\})}-a_j^{(A\cup\{i,k\})})\delta_k m_j^{[A\cup \{i,k\}]}(s)\|_{2+\ep}^{2+\ep}ds\\
    \le
    &\int_0^t(\sum_{j\notin A\cup \{i,k\}}\|(a_j^{(A\cup\{i\})}-a_j^{(A\cup\{i,k\})})\|_{4+2\ep}\|\delta_k m_j^{[A\cup \{i,k\}]}(s)\|_{4+2\ep})^{2+\ep}ds\\
    \le&
    \int_0^t(\sum_{j\notin A\cup\{i,k\}}\frac{C_{t,\ep}}{n})ds\le C_{t,\ep}.
\end{align*}
The contribution from the case $j=k$ is also of order $O(1)$.
Therefore,
\begin{align*}
    J_1\le& C_{t,\ep}+\frac{\ep(1+\ep)}{2}\int_{0}^t\E|X_s|^{2+\ep}ds\\
    &
    +
    (1+\ep)2^{1+\ep}\sup_{\substack{k\notin A\cup \{i\}\\\sigma_i=\pm1}}\int_0^t\|\sum_{j\notin A\cup\{i,k\}}a_j^{(A\cup\{i,k\})}\delta_km_j^{[A\cup\{i,k\}]}(s)\|_{2+\ep}^{2+\ep}ds.
\end{align*}
For $J_2$, from Young's inequality and \cref{eq:mij}, we have
\begin{align*}
    J_2
    \le&(1+\ep)\int_0^t\E|X_s|^{2+\ep}ds
    +
    \sum_{k\notin A\cup\{i\}}\int_0^t\E|\sum_{j\notin A\cup \{i\}}a_j^{(A\cup\{i\})}\delta_i\!\left( m_k^{[A\cup\{i\}]} m_{kj}^{[A\cup\{i\}]} \right)|^{2+\ep}(s)\frac{ds}{n}\\
    \le&(1+\ep)\int_0^t\E|X_s|^{2+\ep}ds
    +
    \sum_{k\notin A\cup\{i\}}\int_0^t\E|\sum_{j\notin A\cup \{i\}}a_j^{(A\cup\{i\})}\delta_k m_{j}^{[A\cup\{i,k\}]} |^{2+\ep}(s)\frac{ds}{n}\\
    \le&(1+\ep)\int_0^t\E|X_s|^{2+\ep}ds
    +
   2^{1+\ep}\sup_{\substack{k\notin A\cup\{i\}\\\sigma_{i}=\pm1}}\int_0^t\E|\sum_{j\notin A\cup \{i\}}a_j^{(A\cup\{i,k\})}\delta_km_{j}^{[A\cup\{i,k\}]}|^{2+\ep}(s)ds\\
   &+
   2^{1+\ep}\sup_{\substack{k\notin A\cup\{i\}\\\sigma_i=\pm1}}\int_0^t\|\sum_{j\notin A\cup\{i\}}(a_j^{(A\cup\{i\})}-a_j^{(A\cup\{i,k\})})\delta_km_j^{[A\cup\{i,k\}]}(s)\|_{2+\ep}^{2+\ep}ds\\
   \le&C_{t,\ep}+(1+\ep)\int_0^t\E|X_s|^{2+\ep}ds
    +
   2^{1+\ep}\sup_{\substack{k\notin A\cup\{i\}\\\sigma_{i}=\pm1}}\int_0^t\E|\sum_{j\notin A\cup \{i,k\}}a_j^{(A\cup\{i,k\})}\delta_km_{j}^{[A\cup\{i,k\}]}|^{2+\ep}(s)ds,
\end{align*}
where the last inequality comes from the same argument as for $J_1$.
Combining the bounds on $J_1,J_2$, we obtain
\bes{
\E|X_t|^{2+\ep}\le & C_{t,\ep}+(1+\frac{\ep}{2})(1+\ep)\int_0^t\E|X_s|^{2+\ep}ds\\
&+(2+\ep)2^{1+\ep}\sup_{\substack{k\notin A\cup\{i\}\\\sigma_i=\pm 1}}\int_0^t\E|\sum_{j\notin A\cup\{i,k\}}a_j^{(A\cup\{i,k\})}\delta_km_{j}^{[A\cup\{i,k\}]}|^{2+\ep}(s)ds.
}
The lemma is proved using the same argument at the end of the proof of \cref{two-point} via Gronwall's inequality and iteration.
\qed

\paragraph{Proof of \cref{lemma-coeffi-A}.}
Recall the definition of $a_j$ from \cref{Delta-error-higher-order}.
From \cref{eq:sumjkdeltai}, we have
\begin{align*}
     X_t:=&\sum_{j=1}^{n} \left[ \sum_{l=1}^{n} m_{jl}^{(A\cup\{i\})} \right]\delta_i\left[\sum_{k=1}^{n}m_{jk}^{[A\cup\{i\}]}\right](t)
     =\sum_{j=1}^{n}\sum_{k=1}^{n}a_{j}^{(A\cup\{i\})}\delta_i m_{jk}^{[A\cup\{i\}]}\\
      =&\sum_{l\notin A}\int_0^t\varepsilon_i\sum_{j=1}^{n}\sum_{k=1}^{n}a_{j}^{(A\cup\{i\})}m_{jkl}^{[A\cup\{i\}]}(s)dg_{il}(s)-\sum_{l\notin A}\int_0^t\delta_i\sum_{j=1}^{n}\sum_{k=1}^{n}a_{j}^{(A\cup\{i\})}(m_{kl}^{[A\cup\{i\}]}m_{jl}^{A\cup\{i\}})(s)\frac{ds}{n}\\
      &-\sum_{l\notin A}\int_0^t\delta_i\sum_{j=1}^{n}\sum_{k=1}^{n}a_{j}^{(A\cup\{i\})}(m_{l}^{[A\cup\{i\}]}m_{jkl}^{A\cup\{i\}})(s)\frac{ds}{n}.
\end{align*}
Similarly to bounding \cref{eq:sumjkdeltai}, from It\^o\!\!'s lemma, we have
\begin{align*}
    \E X_t^2\le&\sum_{l\notin A}\int_0^t\E|\varepsilon_i\sum_{j=1}^{n}\sum_{k=1}^{n}a_{j}^{(A\cup\{i\})}m_{jkl}^{[A\cup\{i\}]}|^2(s) \frac{ds}{n}\\
    &+2\sum_{l\notin A}\int_0^t\E |X_s||\delta_i\sum_{j=1}^{n}\sum_{k=1}^{n}a_{j}^{(A\cup\{i\})}(m_{kl}^{[A\cup\{i\}]}m_{jl}^{A\cup\{i\}})|(s)\frac{ds}{n}\\
    &+2\sum_{l\notin A}\int_0^t\E|X_s||\delta_i\sum_{j=1}^{n}\sum_{k=1}^{n}a_{j}^{(A\cup\{i\})}(m_{l}^{[A\cup\{i\}]}m_{jkl}^{A\cup\{i\}})|(s)\frac{ds}{n}\\
    =:&J_1+J_2+J_3.
\end{align*}
For $J_1$, using \cref{eq:mijk} in the second line below, we obtain
\begin{align*}
    J_1
    \le&\sup_{\substack{l\notin A\cup \{i\}\\\sigma_i=\pm1}}\int_0^t\E|\sum_{j=1}^n\sum_{k=1}^na_j^{(A\cup\{i\})}m_{jkl}^{[A\cup\{i\}]}|^2(s)ds\\
    \le&\sup_{\substack{l\notin  A\cup\{i\}\\\sigma_i=\pm1}}4\int_0^t\E|\sum_{j=1}^n\sum_{k=1}^na_j^{A\cup\{i, l\}}\delta_lm_{jk}^{[A\cup\{i,l\}]}|^2(s)ds\\
    &+\sup_{\substack{l\notin  A\cup\{i\}\\\sigma_i=\pm1}}16\int_0^t\E|(\sum_{k=1}^nm_{kl}^{[A\cup\{i\}]})(\sum_{j=1}^na_j^{(A\cup \{i,l\})}\delta_lm_j^{[A\cup \{i,l\}]})|^2(s)ds\\
    &+\sup_{\substack{l\notin  A\cup \{i\}\\\sigma_i=\pm1}}2\int_0^t\E|\sum_{j=1}^n\sum_{k=1}^n(a_j^{(A\cup \{i\})}-a_j^{(A\cup \{i,l\})})m_{jkl}^{[A\cup \{i\}]}|^2(s)ds\\
    \le&\sup_{\substack{l\notin  A\cup \{i\}\\\sigma_i=\pm1}}4\int_0^t\E|\sum_{j=1}^n\sum_{k=1}^na_j^{(A\cup \{i,l\})}\delta_lm_{jk}^{[A\cup \{i,l\}]}|^2(s)ds\\
    &+\sup_{\substack{l\notin  A\cup \{i\}\\\sigma_i=\pm1}}16\int_0^t\|\sum_{k=1}^nm_{kl}^{[A\cup \{i\}]}\|_4^2\|\sum_{j=1}^na_j^{(A\cup \{i,l\})}\delta_lm_j^{[A\cup \{i,l\}]}\|_4^2(s)ds+C_t\\
    \le&C_t+\sup_{\substack{l\notin  A\cup \{i\}\\\sigma_i=\pm1}}4\int_0^t\E|\sum_{j=1}^n\sum_{k=1}^na_j^{(A\cup \{i,l\})}\delta_lm_{jk}^{[A\cup \{i,l\}]}|^2(s)ds,
\end{align*}
where the last inequality is from \cref{two-point} and \cref{two-point-coef} (with $\ep=2$) and the second to last inequality is because
\begin{align*}
    &\E|\sum_{j=1}^n\sum_{k=1}^n(a_j^{(A\cup \{i\})}-a_j^{(A\cup \{i,l\})})m_{jkl}^{[A\cup \{i\}]}|^2(s)\\
    \le&\sum_{j_1,j_2}\E\Bigg\{\big|(a_{j_1}^{(A\cup \{i\})}-a_{j_1}^{(A\cup \{i, l\})})(\sum_km_{j_1kl}^{[A\cup \{i,l\}]})\big]\cdot\big[(a_{j_2}^{(A\cup \{i\})}-a_{j_2}^{(A\cup \{i, l\})})(\sum_km_{j_2kl}^{[A\cup \{i,l\}]})\big|\Bigg\}(s)\\
    \le&\sum_{j_1,j_2}\|(a_{j_1}^{(A\cup \{i\})}-a_{j_1}^{(A\cup\{i,k\})})(s)\|_4\|(a_{j_2}^{(A\cup \{i\})}-a_{j_2}^{(A\cup\{i,k\})})(s)\|_4\|\sum_km_{j_1kl}^{[A\cup \{i,l\}]}(s)\|_4\|\sum_km_{j_2kl}^{[A\cup \{i,l\}]}(s)\|_4\\
    \le&C_t. \quad \text{[by \cref{Delta-error-higher-order,three-onesum,two-point}]}
\end{align*}
For $J_2$, we have
\begin{align*}
    J_2
    \le&\int_0^t\E X_s^2ds+\sum_{l\notin A}\int_0^t\E(\delta_i\sum_{j=1}^n\sum_{k=1}^na_j^{(A\cup \{i\})}(m_{kl}^{[A\cup \{i\}]}m_{jl}^{[A\cup \{i\}]}))^2(s)\frac{ds}{n}\\
    \le&\int_0^t\E X_s^2ds+\sum_{l\notin A}\int_0^t\E(\delta_i\sum_{j=1}^n\sum_{k=1}^na_j^{(A\cup \{i\})}(m_{kl}^{[A\cup \{i\}]}\delta_l m_{j}^{[A\cup \{i,l\}]}))^2(s)\frac{ds}{n}\\
    \le&\int_0^t\E X_s^2ds+2\sup_{\substack{l\notin A\\\sigma_i=\pm1}}\int_0^t\|\sum_{k=1}^nm_{kl}^{[A\cup \{i\}]}(s)\|_4^2\|\sum_{j=1}^na_j^{(A\cup \{i,  l\})}\delta_lm_{j}^{[A\cup \{i,l\}]}(s)\|_4^2ds\\
    &+2\sup_{\substack{l\notin A\\\sigma_i=\pm1}}\int_0^t\E(\sum_{k=1}^nm_{kl}^{[A\cup \{i\}]}\sum_{j=1}^n(a_j^{(A\cup \{i\})}-a_j^{(A\cup \{i,l\})})\delta_lm_j^{[A\cup \{i,l\}]})^2(s)ds\\
    \le&\int_0^t\E X_s^2ds+C_t,
\end{align*}
where in the last inequality, we used \cref{two-point,two-point-coef} and
\begin{align*}
   &\int_0^t\E(\sum_{k=1}^nm_{kl}^{[A\cup \{i\}]}\sum_{j=1}^n(a_j^{(A\cup \{i\})}-a_j^{(A\cup \{i,l\})})\delta_lm_j^{[A\cup \{i,l\}]})^2(s)ds\\
    \le&\int_0^t\Big(\sum_{j\notin A}\|\sum_{k=1}^nm_{kl}^{[A\cup \{i\}]}(a_j^{(A\cup \{i\})}-a_j^{(A\cup \{i,l\})})\delta_{l}m_j^{[A\cup \{i,l\}]}(s)\|_2\Big)^2ds\\
    \le&\int_0^t\Big(\sum_{j\notin A}\|\sum_{k=1}^nm_{kl}^{[A\cup \{i\}]}(s)\|_8\|\delta_lm_j^{[A\cup \{i,l\}]}(s)\|_8\|a_j^{(A\cup \{i\})}-a_j^{(A\cup \{i,  l\})}(s)\|_4\Big)^2ds
    \le C_t.
\end{align*}
Similarly to bounding $J_1$,  we have
\begin{align*}
    J_3\le&\int_0^t\E X_s^2ds+\sup_{\substack{l\notin A\\\sigma_i=\pm1}}\int_0^t\E(\sum_{j=1}^n\sum_{k=1}^na_j^{(A\cup \{i\})}m_{jkl}^{[A\cup \{i\}]})^2(s)ds\\
    \le&C_t+\int_0^t\E X_s^2ds+
    \sup_{\substack{l\notin  A\cup \{i\}\\\sigma_i=\pm1}}4\int_0^t\E|\sum_{j=1}^n\sum_{k=1}^na_j^{(A\cup \{i,l\})}\delta_lm_{jk}^{[A\cup \{i,l\}]}|^2(s)ds.
\end{align*}

Combining the bounds on $J_1$--$J_3$, we obtain
\begin{equation*}
    \E X_t^2\le C_t+2\int_0^t\E X_s^2ds+8
    \sup_{\substack{l\notin  A\cup \{i\}\\\sigma_i=\pm1}}\int_0^t\E|\sum_{j=1}^n\sum_{k=1}^na_j^{(A\cup \{i,l\})}\delta_lm_{jk}^{[A\cup \{i,l\}]}|^2(s)ds.
\end{equation*}
The lemma is proved using the same argument at the end of the proof of \cref{two-point} using Gronwall's inequality and iteration.
\qed


\section*{Acknowledgements}

Fang X. thanks Yuta Koike for teaching him  the martingale embedding via F\"ollmer process by \cite{eldan2020clt}.
Fang X. was partially supported by Hong Kong RGC GRF 14304822, 14303423, 14302124 and a CUHK direct grant.

\bibliographystyle{apalike}
\bibliography{reference}









\end{document}